\documentclass[10.5pt]{amsart}
\usepackage[english]{babel}
\usepackage{amsmath}
\usepackage{amssymb}
\usepackage{amsthm}
\usepackage{libertine}
\usepackage[all]{xy}
\usepackage{datetime}
 \newtheorem{thm}{Theorem}[section]
 \newtheorem{cor}[thm]{Corollary}
 \newtheorem{lem}[thm]{Lemma}
 \newtheorem{prop}[thm]{Proposition}
  
 \theoremstyle{definition}
 \newtheorem{defn}[thm]{Definition}
 \theoremstyle{remark}
 \newtheorem{rem}[thm]{Remark}
 
 \numberwithin{equation}{section}

\newcommand{\scal}[1]{\left<#1\right>}
\newcommand{\Hq}{\mathbb H}
\newcommand{\Sq}{\mathbb S}

\newcommand{\R}{\mathbb{R}}      

\newcommand{\C}{\mathbb{C}}

\title[On the Bargmann-Fock-Fueter and Bergman-Fueter integral transforms ]{On the Bargmann-Fock-Fueter and Bergman-Fueter integral transforms }

\author[K. Diki]{Kamal Diki}
\thanks{Kamal Diki : Marie Sklodowska-Curie fellow of the Istituto Nazionale di Alta Matematica }
\address{(KD) Politecnico di
Milano\\Dipartimento di Matematica\\Via E. Bonardi, 9\\20133 Milano,
Italy}
\email{kamal.diki@polimi.it}

\author[S. Krausshar]{Rolf S\"{o}ren Krausshar}
\address{(RSK) Fachgebiet Mathematik, Erziehungswissenschaftliche Fakultat, Universtat Erfurt, Nordh\"{a}user Str. 63, D-99089 Erfurt,
Germany }
\email{soeren.krausshar@uni-erfurt.de}

\author[I. Sabadini]{Irene Sabadini}
\address{(IS) Politecnico di
Milano\\Dipartimento di Matematica\\Via E. Bonardi, 9\\20133 Milano\\Italy}
\email{irene.sabadini@polimi.it}

\begin{document}
\maketitle
\begin{abstract}
This paper deals with some special integral transforms of Bargmann-Fock type in the setting of quaternionic valued slice hyperholomorphic and Cauchy-Fueter regular functions. The construction is based on the well-known Fueter mapping theorem. In particular, starting with the normalized Hermite functions we can construct an Appell system of quaternionic regular polynomials. The ranges of such integral transforms are quaternionic reproducing kernel Hilbert spaces of regular functions. New integral representations and generating functions in this quaternionic setting are obtained in both the Fock and Bergman cases.
 \end{abstract}



\section{Introduction}
In 1880, the French mathematician Appell introduced a new class of polynomial sequences generalizing the well-known property satisfied by the classical monomials with respect to the real derivative, namely $\displaystyle \frac{d}{dx}P_n=nP_{n-1}$ see \cite{Appell1880}. A polynomial sequence $\lbrace{P_n}\rbrace_{n\geq 0}$ of degree $n$ satisfying such an identity with respect to a derivative is called an Appell set or an Appell sequence. Other significant and interesting examples of this class of polynomials are the well-known Hermite, Bernoulli and Euler polynomials. In \cite{Chihara1978,Sheffer} the authors followed a different approach to define an Appell set by requesting the identity $\displaystyle \frac{d}{dx}P_n=P_{n-1}$. The study of Appell sequences has also been performed in the setting of Clifford Analysis with respect to the hypercomplex derivative, see for example \cite{CMF2011,MF2007,Pena2011}.
Moreover, in the recent papers \cite{CSS2016,PSS2016} the authors introduced some special modules of monogenic functions of Bargmann-type in Clifford Analysis. This line of research opens some new research directions on Bargmann-Fock spaces and associated transforms in the setting of Clifford Analysis.
In this paper, we construct an Appell sequence of spherical monogenics in the right Fueter-Bargmann space over quaternions, denoted by $\mathcal{RB}(\Hq)$, and consisting of quaternionic Fueter regular functions that are square integrable with respect to the Gaussian measure. The main tool that we use is the Fueter mapping theorem which relates slice hyperholomorphic functions to Fueter regular ones through the Laplacian. More precisely, we apply the Fueter mapping on a special basis of the slice hyperholomorphic Fock space constructed in \cite{AlpayColomboSabadini2014} and obtain a set of homogeneous monogenic polynomials in the right  monogenic Bargmann space over the quaternions. This allows us to construct on the standard Hilbert space on the real line the so called Bargmann-Fock-Fueter integral transform whose range is a quaternionic reproducing kernel Hilbert space of Cauchy-Fueter regular functions. In particular, we believe that this paper may give a partial answer to Remark 4.6 in \cite{KMNQ2016} about Clifford coherent state transforms using the Fueter mapping theorem in the setting of quaternions. \\ In order to present our results, we collect some basic definitions and preliminaries in Section 2. In Section 3, we study how the Fueter mapping acts on special basis elements of the slice hyperholomorphic Fock space. Then, we show that the obtained polynomials constitute an Appell set in the Bargmann space of Cauchy-Fueter regular functions over the quaternions. In Section 4, the Fock-Fueter kernel is discussed and the Bargmann-Fock-Fueter integral transform is introduced and studied. Section 5 deals with explicit formulas for the slice hyperholomorphic Bergman kernel in some specific cases. Finally, in the last Section we treat a similar integral transform in the case of the Bergman spaces of slice hyperholomorphic functions on the unit ball, on the half space and on the unit half-ball on quaternions.
\section{Preliminaries}
Let $\lbrace{e_1,e_2, . . . , e_m}\rbrace$ be an orthonormal basis of the Euclidean vector space $\R^m$ in which we introduce a non-commutative product defined by the following multiplication law
$$e_ke_s + e_se_k = -2\delta_{k,s}, k, s= 1, . . . , m$$
where $\delta_{k,s}$ is the Kronecker symbol. The set $$\lbrace{e_A : A \subset\lbrace{1, . . . , m}\rbrace \text{ with }
e_A = e_{h_1}e_{h_2}...e_{h_r}, 1 \leq h_1 < ...< h_r \leq m, e_{\emptyset} = 1}\rbrace $$
forms a basis of the $2^m$-dimensional Clifford algebra $\R_m$ over $\R$. Let $\R^{m+1}$ be embedded in $\R_m$ by identifying
$(x_0, x_1,..., x_m) \in \R^{m+1}$ with the paravector $x=x_0+\underline{x}\in \R_m$. The conjugate of $x$ is given by $\bar{x} = x_0-\underline{x}$
and the norm $\vert{x}\vert$ of $x$ is defined by $\vert{x}\vert^2=x_0^2+...+x_m^2$. \\
For $m\geq 1$, the Euclidean Dirac operator on $\R^m$ is given by $$\displaystyle\partial_{\underline{x}}=\sum_{k=1}^me_k\partial_{x_k}.$$  The generalized Cauchy-Riemann operator (also known as Weyl operator) and its conjugate in $\R^{m+1}$ are given respectively by $$\displaystyle\partial:=\partial_{x_0}+\partial_{\underline{x}} \text{ and }  \displaystyle\overline{\partial}:=\partial_{x_0}-\partial_{\underline{x}}.$$ Notice that $$\overline{\partial}\partial=\partial\overline{\partial}=\Delta_{\R^{m+1}}$$ where $\Delta_{\R^{m+1}}$ stands for the usual Laplacian on the Euclidean space $\R^{m+1}$. Real differentiable functions on an open subset of $\R^{m+1}$ taking their values in $\R_m$ that are in the kernel of the generalized Cauchy-Riemann operator are called left monogenic or monogenic, for short. Moreover, for a monogenic function $f$ we have the following Leibniz rule, see e.g. \cite{Pena2008}
\begin{equation} \label{lb}
\displaystyle\partial_{\underline{x}}(\underline{x}f)=-mf-\underline{x}\partial_{\underline{x}}f-2\sum_{l=1}^mx_l\partial_{x_l}f.
 \end{equation} The latter formula will be very important for our calculations. In the particular case of quaternions the generalized Cauchy-Riemann operator in $\R^{m+1}$ becomes the Cauchy-Fueter operator and this leads to the theory of quaternionic Fueter regular functions. Specifically, we have the following
\begin{defn} Let $U\subset \Hq$ be an open set and let $f:U\longrightarrow \Hq$ be a real differentiable function. We say that $f$ is (left) Fueter regular or regular for short on $U$ if $$\partial f(q):=\displaystyle\left(\frac{\partial}{\partial x_0}+i\frac{\partial}{\partial x_1}+j\frac{\partial}{\partial x_2}+k\frac{\partial}{\partial x_3}\right)f(q)=0, \forall q\in U.$$
\\
The right linear space of Fueter regular functions will be denoted by $\mathcal{R}(U)$.
\end{defn}
 As customary, a Fueter regular polynomial of degree $k$ is called a quaternionic spherical monogenic of degree $k$.
\

For more details about the theory of quaternionic Fueter regular functions we refer the reader to, e.g. \cite{CSSS2004,GHS}. In 2006 a new approach to quaternionic regular functions was introduced  and then extensively studied in several directions, and it is nowadays widely developed, see \cite {ACS_book,CSS,CSS_book, GentiliSS}.

 We briefly revise here the basics of this function theory. Let  $\mathbb{S}=\lbrace{q\in{\Hq};q^2=-1}\rbrace$ be the unit sphere of imaginary units in $\mathbb H$.
   Note that any $q\in \Hq\setminus \R$ can be rewritten in a unique way as $q=x+I y$ for some real numbers $x$ and $y>0$, and imaginary unit $I\in \mathbb{S}$. For every given $I\in{\mathbb{S}}$, we define $\C_I = \mathbb{R}+\mathbb{R}I.$
It is isomorphic to the complex plane $\C$ so that it can be considered as a complex plane in $\Hq$ passing through $0$, $1$ and $I$. Their union is the whole space of quaternions $$\Hq=\underset{I\in{\mathbb{S}}}{\cup}\C_I =\underset{I\in{\mathbb{S}}}{\cup}(\mathbb{R}+\mathbb{R}I).$$ Then, we have the following
\begin{defn}
A real differentiable function $f: \Omega \longrightarrow \Hq$, on a given domain $\Omega\subset \Hq$, is said to be a (left) slice hyperholomorphic function if, for every $I\in \Sq$, the restriction $f_I$ to $\C_{I}=\R+I\R$, with variable $q=x+Iy$, is holomorphic on $\Omega_I := \Omega \cap \C_I$, that is it has continuous partial derivatives with respect to $x$ and $y$ and the function
$\overline{\partial_I} f : \Omega_I \longrightarrow \Hq$ defined by
$$
\overline{\partial_I} f(x+Iy):=
\dfrac{1}{2}\left(\frac{\partial }{\partial x}+I\frac{\partial }{\partial y}\right)f_I(x+yI)
$$
vanishes identically on $\Omega_I$. The set of slice hyperholomorphic functions will be denoted by $\mathcal{SR}(\Omega)$.
\end{defn}
An important result of this function theory is the following:
\begin{thm}[Representation formula]
Let $f\in{\mathcal{SR}(\Hq)}$ and $I,J\in{\mathbb{S}}$.\\
For all $q=x+yI\in{\Hq}$, we have
 $$
\displaystyle f(x+yI)= \alpha(x,y)+I\beta(x,y)
$$  where
$
\displaystyle \alpha(x,y)= \dfrac{1}{2}[f(x+yJ)+f(x-yJ)]$ and $\displaystyle\beta(x,y)=\frac{J}{2}[f(x-yJ)-f(x+yJ)]
$ Moreover, $\alpha$ and $\beta$ satisfy the Cauchy-Riemann conditions.

\end{thm}
The two theories of slice hyperholomorphic and Fueter regular functions are related  by the Fueter mapping theorem, see \cite{ColomboSabadiniSommen2010}. We briefly recall below the variation of this result that we will use later and we refer the reader to \cite{Qian2014} for more information.
\begin{thm}[Fueter mapping theorem, see \cite{ColomboSabadiniSommen2010}]
Let $U$ be an axially symmetric set in $\mathbb H$ and let $f:U\subset \Hq\longrightarrow \Hq$ be a slice regular function of the form $f(x+yI)=\alpha(x,y)+I\beta(x,y)$,where $\alpha(x,y)$ and $\beta(x,y)$ are quaternionic-valued functions such that $\alpha(x,-y)=\alpha(x,y)$, $\beta(x,-y)=-\beta(x,y)$ and satisfying the Cauchy-Riemann system. Then, the function $$\overset{\sim}f(x_0,\vert{\underline{q}}\vert)=\displaystyle \Delta\left(\alpha(x_0,\vert{\underline{q}}\vert)+\frac{\underline{q}}{\vert{\underline{q}}\vert}\beta(x_0,\vert{\underline{q}}\vert)\right)$$
is Fueter regular.
\end{thm}
\begin{rem}
The representation formula combined with this version of the Fueter mapping theorem shows that to each slice regular function $f$ on $\Hq$ is associated a Cauchy-Fueter regular function given by $\overset{\sim}f(q)=\Delta f(q)$ where $\Delta$ denotes the Laplacian on the Euclidean space $\R^4$. Below, we will also denote the Fueter mapping by $$\tau:\mathcal{SR}(U)\rightarrow\mathcal{R}(U), \text{ } f\longmapsto \tau(f)=\overset{\sim}f.$$
\end{rem}
 Finally, we recall the definition of the slice hyperholomorphic Fock space $\mathcal{F}_{Slice}(\Hq)$ and the right Fueter-Bargmann space $\mathcal{RB}(\Hq)$ in the context of quaternions. These two spaces were introduced in the quaternionic and Clifford Analysis setting in the papers \cite{AlpayColomboSabadini2014,CSS2016,PSS2016}.
\paragraph{}
The authors of \cite{AlpayColomboSabadini2014} defined the slice hyperholomorphic quaternionic Fock space $\mathcal{F}_{Slice}(\Hq)$, for a given $I\in{\mathbb{S}}$  to be
$$\mathcal{F}_{Slice}(\Hq):=\lbrace{f\in{\mathcal{SR}(\Hq); \, \displaystyle  \frac{1}{\pi} \int_{\C_I}\vert{f_I(p)}\vert^2 e^{-\vert{p}\vert^2}d\lambda_I(p) <\infty}}\rbrace,$$
 where $f_I = f|_{\C_I}$ and $d\lambda_I(p)=dxdy$ for $p=x+yI$.
This right $\Hq$-vector space is endowed with the following inner product
\begin{equation}\label{spfg}
\scal{f,g}_{\mathcal{F}_{Slice}(\Hq)} = \frac{1}{\pi}\int_{\C_I}\overline{g_I(p)}f_I(p)e^{-\vert{p}\vert^2} d\lambda_I(p),
\end{equation}
 where $f,g\in{\mathcal{F}_{Slice}(\Hq)}$, so that the associated norm is given by
      $$\Vert{f}\Vert_{\mathcal{F}_{Slice}(\Hq)}^2= \frac{1}{\pi}\int_{\C_I}\vert{f_I(p)}\vert^2 e^{-\vert{p}\vert^2}d\lambda_I(p).$$
 Note that all these norms, which depend on $I\in\mathbb S$, are equivalent. In fact, it was shown in \cite{AlpayColomboSabadini2014} that $\mathcal{F}_{Slice}(\Hq)$ does not depend on the choice of the imaginary unit $I\in \mathbb{S}$. Moreover, the monomials $f_n(p):=p^n$; $n=0,1,2, \cdots,$ form an orthogonal basis of the space
 with
 \begin{equation}\label{spmonomials}
 \scal{f_m,f_n}_{\mathcal{F}_{Slice}(\Hq)}  = m!\delta_{m,n}.
\end{equation}
On the other hand, in Section 5 of \cite{PSS2016} and Section 3 of \cite{CSS2016} the real monogenic Bargmann module on the Euclidean space $\R^m$ was defined to be the module consisting of solutions of the $s$-th power of the Dirac operator that are square integrable on $\R^m$ with respect to a Gaussian measure. In this paper we work with a similar definition for the quaternions by replacing the $s$-th power of the Dirac operator by the Cauchy-Fueter operator. So, we call the right Fueter-Bargmann space on quaternions  the space defined by

$$\mathcal{RB}(\Hq):=\lbrace{f\in{\mathcal{R}(\Hq); \, \displaystyle  \frac{1}{\pi^2} \int_{\Hq}\vert{f(q)}\vert^2 e^{-\vert{q}\vert^2}d\lambda(q) <\infty}}\rbrace,$$
where $d\lambda$ denotes the usual Lebesgue measure on the Euclidean vector space $\R^4$.
\section{The action of the Fueter mapping on the quaternionic monomials}
The main goal of this section is to apply the Fueter mapping on the quaternionic monomials forming an orthogonal basis of the slice hyperholomorphic Fock space  $\mathcal{F}_{Slice}(\Hq)$ and to get an Appell set of $\mathcal{RB}(\Hq)$. A different proof of this result using Cauchy-Kowalevski extension arguments can be found in \cite{NG2009}.
\\ \\ First, we need a lemma that describes the action of the Cauchy-Fueter operator on the quaternionic monomials $f_n(q)=q^n$:
\begin{lem}[see \cite{Begehr1999}] \label{DiracAct}
For all $n\geq 2$, we have
$$\partial f_n(q)=-2\displaystyle\sum_{k=1}^{n}q^{n-k}\overline{q}^{k-1}.$$
\end{lem}
Then, we prove the following
\begin{thm} \label{FueterAct}
For all $n\geq 2$, we have
$$\tau [f_n](q)=\overset{\sim}f_n(q)=-4\displaystyle\sum_{k=1}^{n-1}(n-k)q^{n-k-1}\overline{q}^{k-1}.$$
\end{thm}
\begin{proof} Let $q=x_0+x_1i+x_2j+x_3k$, thanks to the quaternions  multiplication rules we have $$f_2(q)=q^2=x_0^2-x_1^2-x_2^2-x_3^2+2x_0(x_1i+x_2j+x_3k)$$ It is easy to check that $\overset{\sim}f_2(q)=-4$, so the formula holds for $n=2$.\\
  Let $n\geq 2$. We suppose the proposition is true for $n$ and we show that $$\overset{\sim}f_{n+1}(q)=-4\displaystyle\sum_{l=1}^{n}(n+1-k)q^{n-k}\overline{q}^{k-1}.$$
  Indeed, we have $\overset{\sim}f_{n+1}=\Delta f_{n+1}$ and $f_{n+1}(q)=qf_n(q).$ Therefore, applying the classical Leibniz rule we get the following system

  \[
\left \{
\begin{array}{c @{=} c}
    \displaystyle\frac{\partial^2}{\partial x_{0}^2}f_{n+1}(q) &\displaystyle 2\frac{\partial}{\partial x_0}f_n(q)+q\frac{\partial^2}{\partial x_{0}^2}f_n(q) \\
    \displaystyle \frac{\partial^2}{\partial x_{1}^2}f_{n+1}(q) & \displaystyle 2i\frac{\partial}{\partial x_1}f_n(q)+q\frac{\partial^2}{\partial x_{1}^2}f_n(q) \\
    \displaystyle  \frac{\partial^2}{\partial x_{2}^2}f_{n+1}(q) & \displaystyle 2j\frac{\partial}{\partial x_2}f_n(q)+q\frac{\partial^2}{\partial x_{2}^2}f_n(q) \\
      \displaystyle \frac{\partial^2}{\partial x_{3}^2}f_{n+1}(q) & \displaystyle 2k\frac{\partial}{\partial x_3}f_n(q)+q\frac{\partial^2}{\partial x_{3}^2}f_n(q)
\end{array}
\right.
\] Thus, by adding both sides of the system we obtain
 $$\overset{\sim}f_{n+1}=\Delta f_{n+1}=2\partial f_n+q\Delta f_n=2\partial f_n+q\overset{\sim}f_{n}$$
Then, thanks to Lemma \ref{DiracAct} combined with the induction hypothesis we obtain $$\overset{\sim}f_{n+1}(q)=-4\displaystyle\sum_{k=1}^{n}(n+1-k)q^{n-k}\overline{q}^{k-1}.$$
This completes the proof.
  \end{proof}

  \begin{prop}
For all $n\geq 2$, we have $\overset{\sim}f_{n}\in \mathcal{RB}(\Hq)$.
  \end{prop}
  \begin{proof}
 First of all, by the Fueter mapping theorem the functions $\overset{\sim}f_{n}$ are monogenic.
 We now show that for $n\geq 2$, we have $$\displaystyle \int_{\Hq}\vert{\overset{\sim}f_{n}(q)}\vert^2e^{-\vert{q}\vert^2}d\lambda(q)<\infty.$$
Indeed, we have \begin{align*}
\displaystyle \left|\overset{\sim}f_{n}(q)\right|&=4 \left|\sum_{k=1}^{n-1}(n-k)q^{n-k-1}\overline{q}^{k-1}\right|
 \\& \leq 4 \sum_{k=1}^{n-1}(n-k)\left|q\right|^{n-2}
 .
 \end{align*}
 And since $\displaystyle\sum_{k=1}^{n-1}(n-k)=\frac{(n-1)n}{2}$ we get the following estimate $$\vert{\overset{\sim}f_n(q)}\vert\leq 2(n-1)n \vert{q}\vert^{n-2}$$
 Hence, for all $n\geq 2,$ we have $\Vert{\overset{\sim}f_n}\Vert_{\mathcal{RB}(\Hq)}\leq 2n(n-1)\Vert{f_{n-2}}\Vert_{L^2(\Hq)}$. The proof is completed since the quaternionic monomials are square integrable with respect to the Gaussian measure on $\Hq$.

  \end{proof}
  \begin{rem}
  Let $n\geq 2$ and $k\geq 0$. Then
 \begin{enumerate}

 \item The functions $\overset{\sim}f_n$ are spherical monogenics of degree $n-2$.
 \item $\overset{\sim}f_{n+1}=2\partial f_n+q\overset{\sim}f_{n}$.
 \item  $\overset{\sim}f_{k+2}(q)=-4\displaystyle\sum_{j=0}^{k}(k+1-j)q^{k-j}\overline{q}^{j}$.
 \end{enumerate}
  \end{rem}
 As a consequence we obtain an Appell set of spherical monogenics in
$\mathcal{RB}(\Hq)$. To prove this fact we need some preliminary lemmas.
\begin{lem} \label{l1} Let $f:\Hq\longrightarrow \Hq$ be a Fueter regular function. Then, $$\displaystyle\overline{\partial}(qf)=4f+2\sum_{l=0}^3x_l\partial_{x_l}f.$$
\end{lem}
\begin{proof}
Notice that for the particular case of quaternions the Leibniz rule given by  \eqref{lb} correspond to $m=3$. Then, if we write $q=x_0+\underline{x}$ with $\underline{x}=x_1i+x_2j+x_3k$ we obtain \begin{equation} \label{Leibniz} \displaystyle\partial_{\underline{x}}(\underline{x}f)=-3f-\underline{x}\partial_{\underline{x}}f-2\sum_{l=1}^3x_l\partial_{x_l}f. \end{equation}
Morever, we have $$\displaystyle\overline{\partial}(qf)=f+x_0\overline{\partial}f+\underline{x}\partial_{x_0}f-\partial_{\underline{x}}(\underline{x}f).$$
So, thanks to \eqref{Leibniz} we obtain $$\displaystyle\overline{\partial}(qf)=4f+x_0\overline{\partial}f +\underline{x}\partial f+2\sum_{l=1}^3x_l\partial_{x_l}f.$$  It is easy to see that  $\overline{\partial}f=2\partial_{x_0}f$. Moreover, if $f$ is regular then $\partial f=0$ which completes the proof.
\end{proof}
Let us consider the Euler operator $$\displaystyle E_q:=\sum_{l=0}^3x_l\partial_{x_l}.$$ We have:
\begin{lem}  \label{Euleraction}
Let $h\geq 2$ and  $0\leq s \leq h$. Then
$$E_q(q^{h-s}\overline{q}^{s})=hq^{h-s}\overline{q}^{s}.$$
\end{lem}
\begin{proof}
Note that for all $l\geq 0$, we have
\begin{equation} \label{x1q}
\displaystyle\frac{d}{dx_1}(q^l)=iq^{l-1}+qiq^{l-2}+q^{2}iq^{l-3}+...+q^{l-1}i
\end{equation}
and
\begin{equation} \label{x1qbar}
\displaystyle\frac{d}{dx_1}(\bar{q}^l)=-i\bar{q}^{l-1}-\bar{q}i\bar{q}^{l-2}-\bar{q}^{2}i\bar{q}^{l-3}-...-\bar{q}^{l-1}i.
\end{equation}
We have analogous relations for $\displaystyle\frac{d}{dx_2}(q^l),\frac{d}{dx_3}(q^l)$ and $\displaystyle\frac{d}{dx_2}(\bar{q}^l),\frac{d}{dx_3}(\bar{q}^l).$ Now observe that by the classical Leibniz rule we have $$\displaystyle\frac{d}{dx_0}(q^{h-s}\bar{q}^s)=sq^{h-s}\bar{q}^{s-1}+(h-s)q^{h-s-1}\bar{q}^{s}.$$
On the other hand, applying the Leibniz rule we also have $$\displaystyle\frac{d}{dx_1}(q^{h-s}\bar{q}^s)=q^{h-s}\frac{d}{dx_1}(\bar{q}^s)+\frac{d}{dx_1}(q^{h-s})\bar{q}^s.$$
Therefore, we use the formulas \eqref{x1q}, \eqref{x1qbar} and those ones  with respect to all other derivatives to compute $\displaystyle\frac{d}{dx_1}(q^{h-s}\bar{q}^s), \displaystyle\frac{d}{dx_2}(q^{h-s}\bar{q}^s)$ and $\displaystyle\frac{d}{dx_3}(q^{h-s}\bar{q}^s)$.
Then, by standard  computations we obtain the result.
\end{proof}
\begin{lem} \label{l2}
For all $k\geq 1,$
$$\displaystyle\overline{\partial}\overset{\sim}f_{k+2}=2(k+2)\overset{\sim}f_{k+1}.$$
\end{lem}
\begin{proof}
Direct computations show that the formula holds for $k=1$ and $k=2$. \\
Let $k\geq 2,$ we can just prove $\displaystyle\overline{\partial}\overset{\sim}f_{k+3}=2(k+3)\overset{\sim}f_{k+2}$. Indeed, we have $$\overset{\sim}f_{k+3}=2\partial f_{k+2}+q\overset{\sim}f_{k+2}$$ Then, we apply the conjugate of the Cauchy-Fueter operator on both sides of the latter equality and we use the fact that $\overline{\partial}\partial=\partial\overline{\partial}=\Delta$ to get
\begin{equation}\label{ME}\overline{\partial}\overset{\sim}f_{k+3}=2\overset{\sim}f_{k+2}
+\overline{\partial}(q\overset{\sim}f_{k+2}).\end{equation} Let us calculate $\overline{\partial}(q\overset{\sim}f_{k+2})$.

Since $\overset{\sim}f_{k+2}$ is Fueter regular and in view of Lemma \ref{l1} we have  \begin{equation}\label{leibniz2}\displaystyle\overline{\partial}(q\overset{\sim}f_{k+2})=4\overset{\sim}f_{k+2}+2E_q \overset{\sim}f_{k+2},\end{equation}
and
\begin{equation} \label{Euler}E_q \overset{\sim}f_{k+2}=-4\displaystyle\sum_{s=0}^{k}(k+1-s)E_q(q^{k-s}\overline{q}^{s}).\end{equation}
Hence, we apply Lemma \ref{Euleraction} to obtain
 $$E_q(q^{k-s}\overline{q}^{s})=kq^{k-s}\overline{q}^{s}.$$ \\ Therefore, by replacing in \eqref{Euler} we get $$E_q \overset{\sim}f_{k+2}=k\overset{\sim}f_{k+2}.$$ \\ Finally, we conclude from the equations \eqref{ME} and \eqref{leibniz2} that $$\displaystyle\overline{\partial}\overset{\sim}f_{k+3}=2(k+3)\overset{\sim}f_{k+2}.$$ This concludes the proof.
\end{proof}
 For $k\geq 0$, let us consider the sequence of polynomials defined by $$P_k(q):=\displaystyle \frac{\overset{\sim}f_{k+2}(q)}{(k+2)!}.$$ We prove the following
\begin{thm}
The polynomials $\lbrace{P_k}\rbrace_{k\geq 0}$ form an Appell set of spherical monogenics of degree $k$ in the quaternionic vector space $\mathcal{RB}(\Hq)$.
\end{thm}
\begin{proof}
Clearly, each homogeneous monogenic polynomial of the sequence $\lbrace{P_k}\rbrace_{k\geq 0}$ is exactly of degree $k$ and belongs to $\mathcal{RB}(\Hq)$ since $\overset{\sim}f_{k+2}$ is for all $k\geq 0$. Furthermore, thanks to Lemma \ref{l2} we can easily see that for all $k\geq 1$ we have $\overline{\partial}P_{k}=2P_{k-1}$. It follows that this sequence forms an Appell set in $\mathcal{RB}(\Hq)$, in the sense of \cite{Chihara1978,Sheffer}, with respect to the hypercomplex derivative $\displaystyle\frac{1}{2} \overline{\partial}$.
\end{proof}
\begin{rem}\label{remQ}
Let $k\geq 0$ and set $\displaystyle Q_k:=-\frac{k!}{2} P_k$. Then, we have\begin{equation} \label{Qk1}
\displaystyle Q_k=-\frac{\overset{\sim}f_{k+2}}{2(k+1)(k+2)}
\end{equation} and $$\displaystyle\frac{1}{2}\overline{\partial}Q_k=kQ_{k-1}.$$ Moreover, we can see that the obtained family of polynomials may be expressed in terms of the coefficients used in formulas 5 and 6 in the paper \cite{CMF2011}. Namely, we have
\begin{equation} \label{Qk2}
\displaystyle Q_k(q)=\sum_{j=0}^k T^k_jq^{k-j}\overline{q}^j
\end{equation}

where $$\displaystyle T^k_j:=T^{k}_{j}(3)=\frac{k!}{(3)_k}\frac{(2)_{k-j}(1)_j}{(k-j)!j!}=\frac{2(k-j+1)}{(k+1)(k+2)}$$ and $(a)_n=a(a+1)...(a+n-1)$ is the Pochhammer symbol.
\end{rem}
\section{The Bargmann-Fock-Fueter transform on the quaternions}
In this Section, we study the Bargman-Fock-Fueter transform on the space of quaternions. A similar integral transform was introduced in \cite{CGS2011} making use of the theory of slice hyperholomorphic Bergman spaces on the quaternionic unit ball and the Fueter mapping theorem. To this end, we introduce the Fock-Fueter kernel on the quaternions. Indeed, in \cite{AlpayColomboSabadini2014}, the authors proved that the slice hyperholomorphic Fock space $\mathcal{F}_{Slice}(\Hq)$ is a right quaternionic reproducing kernel Hilbert space whose reproducing kernel is given by the formula $$K_\Hq(p,q):=e_*(p\bar{q})=\displaystyle \sum_{k=0}^\infty\frac{p^k\bar{q}^k}{k!}, \qquad \text{ }\forall (p,q)\in \Hq\times\Hq.$$ Then, we consider the following
\begin{defn}[Fock-Fueter kernel]
The Fock-Fueter kernel $K_\mathcal{F}(q,p)$ is defined by
$$K_\mathcal{F}(q,p):=\tau_q K_\Hq(q,p)=\Delta K_{\Hq}(q,p),\qquad \text{  } \forall (q,p)\in \Hq\times \Hq,$$ where Laplacian $\Delta$ is taken with respect to the variable $q$.
   \end{defn}
   We prove the following
\begin{prop} \label{FKex}
For all $(q,p)\in \Hq\times\Hq, $ we have
$$K_\mathcal{F}(q,p)=\displaystyle -2\sum_{k=0}^\infty\frac{Q_k(q)}{k!}\bar{p}^{k+2},$$
where $Q_k(q)$ are the quaternionic monogenic polynomials defined by \eqref{Qk2}.
\end{prop}
\begin{proof}
Let $(q,p)\in\Hq\times\Hq$, by definition of the Fock-Fueter kernel we have
\[ \begin{split}
 \displaystyle K_\mathcal{F}(q,p)&= \Delta K_{\Hq}(q,p) \\
 &= \Delta \left(\sum_{k=0}^{\infty}\frac{q^k\overline{p}^k}{k!}\right) \\
&= \sum_{k=2}^{\infty}\frac{\Delta (q^k) \overline{p}^k}{k!}.\\
\end{split}
\]

However, thanks to Remark \ref{remQ} we observe that $$\displaystyle \Delta(q^k)=-2(k-1)kQ_{k-2}(q);\text{ } \forall k\geq 2.$$
Therefore, we get \[ \begin{split}
 \displaystyle K_\mathcal{F}(q,p)& =-2 \sum_{k=2}^{\infty}\frac{Q_{k-2}(q)}{(k-2)!}\overline{p}^k \\
&= -2\sum_{k=0}^{\infty}\frac{Q_k(q)}{k!}\overline{p}^{k+2}.\\
\end{split}
\]

\end{proof}
\begin{rem}

For $s\in \Hq,$ let $${\rm {Exp}}(s):=\sum_{k=0}^\infty\frac{Q_k(s)}{k!}$$ be the generalized Cauchy-Fueter regular exponential function considered in the paper \cite{CMF2011}.
 Then, we have  $$K_\mathcal{F}(q,p)=-2p^2 {\rm {Exp}}(pq), \text{  } \forall (q,p)\in \Hq\times\R.$$
\end{rem}
\begin{prop}
 The Fock-Fueter kernel $K_\mathcal{F}(q,p)$ is Cauchy-Fueter regular on $\Hq$ with respect to the variable $q$ and anti-slice entire regular with respect to the variable $p$.
\end{prop}
\begin{proof}
 Note that $K_\mathcal{F}(q,p)$ is Cauchy-Fueter regular on $\Hq$ with respect to the first variable thanks to the Fueter-mapping theorem. On the other hand, for all $p\in\Hq$ we have $$\displaystyle K_\mathcal{F}(q,p)= f_q(p)=\sum_{k=0}^\infty a_k(q)\bar{p}^{k+2} \text{ where } a_k(q)=-2\frac{Q_k(q)}{k!}.$$ Then, it is clear by the series expansion theorem for slice hyperholomorphic functions  that the Fock-Fueter kernel is slice anti-regular with respect to the variable $p$.
\end{proof}
The Fock-Fueter kernel admits the following estimate
\begin{prop}
For all $(q,p)\in \Hq\times\Hq, $ we have
$$|K_\mathcal{F}(q,p)| \leq \displaystyle 2|p|^2 e^{|qp|}.$$

\end{prop}
\begin{proof}
First, observe that for all $k\geq 0$ and $q\in\Hq$ we have
\[ \begin{split}
 \displaystyle |Q_k(q)|& \leq \sum_{j=0}^{k}T_{j}^{k}(3) |q|^k \\
&= |q|^k \frac{2}{(k+1)(k+2)}\sum_{j=0}^{k}(k+1-j)\\
&=|q|^k.\\
\end{split}
\]
Hence, making use of Proposition \ref{FKex} we obtain \[ \begin{split}
 \displaystyle |K_\mathcal{F}(q,p)|& \leq 2\sum_{k=0}^{\infty} \frac{|Q_k(q)|}{k!}|p|^{k+2} \\
&\leq 2|p|^{2} \sum_{k=0}^{\infty} \frac{|qp|^k}{k!}\\
&=2|p|^{2}e^{|qp|}.\\
\end{split}
\]
\end{proof}
In this case we introduce the following definition
  \begin{defn}[Fock-Fueter transform]
  Let $f\in \mathcal{F}_{Slice}(\Hq)$. We define the Fock-Fueter transform of $f$ by $$\breve{f}(q):=\displaystyle \int_{\mathbb{C}_I}K_{\mathcal{F}}(q,p)f(p)d\mu_I(p);$$
  where $K_{\mathcal{F}}$ is the Fock-Fueter kernel, $d\mu_I(p)=\displaystyle\frac{1}{\pi}e^{-|p|^2}d\lambda_I(p)$ and $I\in \Sq$.
   \end{defn}
 Let $L^2_\Hq(\R)$ denote the space of functions $\psi:\R\longrightarrow \Hq$ so that  $$ \displaystyle \Vert{\psi}\Vert_{L^2_\Hq(\R)}^2= \int_\R|\psi(x)|^2dx<\infty.$$ Then, for any $\varphi\in L^2_\Hq(\R)$ its quaternionic Segal-Bargmann transform is defined by $$\mathcal{B}_\Hq \varphi(q):=\displaystyle \int_\R A(q,x)\varphi(x)dx;$$ where the kernel function $A(q,x)$ is given by the formula $$A(q,x)=\pi^{-\frac{1}{4}}e^{-\frac{1}{2}(q^2+x^2)+\sqrt{2}qx}, \qquad \forall (q,x)\in \Hq\times \R.$$ It was shown in \cite{DG1.2017} that $\mathcal{B}_\Hq $ defines an isometry on $L^2_\Hq(\R)$ with range $\mathcal{F}_{Slice}(\Hq)$. Then, for any $\varphi\in L^2_\Hq(\R)$ we set $$f_\varphi:=\mathcal{B}_\Hq \varphi\in \mathcal{F}_{Slice}(\Hq)$$ and consider the associated Fock-Fueter transform $\breve{f_\varphi}$ that we will call the Bargmann-Fock-Fueter transform. We can easily check the following
 \begin{prop} \label{BFk1}
 Let $\varphi\in L^2_\Hq(\R)$, $q\in\Hq$ and $I\in\Sq$. Then, we have $$\breve{f_\varphi}(q):=\displaystyle\int_{\R}\Phi(q,x)\varphi(x)dx;$$
 where $$\Phi(q,x)=\displaystyle\int_{\C_I}K_{\mathcal{F}}(q,p)A(p,x)d\mu_I(p).$$
 \end{prop}
  \begin{proof}
  This follows directly from the definitions of the quaternionic Segal-Bargmann transform and the Fock-Fueter integral transform making use of Fubini's theorem.
   \end{proof}
Now, let us consider the quaternionic regular polynomials defined in Remark \ref{remQ} and which may be written as : \begin{equation}
\displaystyle Q_k(q)=\sum_{j=0}^k T^k_jq^{k-j}\overline{q}^j; \forall q\in\Hq.
\end{equation} Then, we denote the range of the Fueter mapping on the slice hyperholomorphic Fock space by $$ \mathcal{A}(\Hq):=\{ \tau (f); \text{ } f\in \mathcal{F}_{Slice}(\Hq)\}.$$
We have the following sequential characterization of this vector space:

\begin{thm} \label{AHcar}
Let $g\in \mathcal{R}(\Hq)$. Then, $g$ belongs to $ \mathcal{A}(\Hq)$ if and only if the following conditions are satisfied:
\begin{enumerate}
\item[i)] $\forall q\in\Hq, \textbf{ }  g(q)=\displaystyle\sum_{k=0}^\infty Q_k(q)\alpha_k$ where $(\alpha_k)_{k\geq 0}\subset \Hq.$
\item[ii)]$\displaystyle \sum_{k=0}^{\infty}\frac{k!}{(k+1)(k+2)}|\alpha_k|^2<\infty.$
\end{enumerate}
 \end{thm}
 \begin{proof}
 The Fueter mapping theorem gives $\mathcal{A}(\Hq)\subset\mathcal{R}(\Hq).$ Then, we suppose that $g\in \mathcal{A}(\Hq)$, thus $g=\tau(f)$ where $f\in \mathcal{F}_{Slice}(\Hq)$. Then, according to \cite{AlpayColomboSabadini2014} we have $$f(q)=\displaystyle \sum_{k=0}^{\infty}q^kc_k , \text{ with } (c_k)\subset \Hq \text{ and } \Vert f \Vert^2_{\mathcal{F}_{Slice}(\Hq)}=\sum_{k=0}^{\infty}k!|c_k|^2<\infty.$$
 However, we know that $$\tau(1)=\tau(q)=0 \text{ and } \tau(q^k)=-2(k-1)kQ_{k-2}(q),\qquad \forall k\geq 2.$$
 Therefore, we get $$g(q)=\displaystyle\sum_{k=0}^\infty Q_k(q)\alpha_k \text{ with } \alpha_k=-2(k+1)(k+2)c_{k+2}, \qquad \forall k\geq 0,$$ moreover,
 $$\displaystyle \sum_{k=0}^{\infty}\frac{k!}{(k+1)(k+2)}|\alpha_k|^2=4\displaystyle \sum_{k=0}^{\infty}(k+2)!|c_{k+2}|^2\leq 4\Vert f \Vert^2_{\mathcal{F}_{Slice}(\Hq)} <\infty.$$

 Conversely, let us suppose that the conditions i) and ii) hold. Then, we consider the function $$h(q)=\displaystyle \sum_{k=2}^{\infty}q^k\beta_k , \text{ where } \beta_k=-\frac{\alpha_{k-2}}{2(k-1)k} ; \forall k\geq 2.$$
 Thus, we get $g=\tau(h)$ since
 $$Q_{k}(q)=-\frac{\tau(q^{k+2})}{2(k+1)(k+2)}, \qquad \forall k\geq 0.$$
 Moreover, note that we have
 $$\displaystyle \Vert h \Vert_{\mathcal{F}_{Slice}(\Hq)}^2=\sum_{k=2}^{\infty}k!|\beta_k|^2=\frac{1}{4}\sum_{k=0}^{\infty}\frac{k!}{(k+1)(k+2)}|\alpha_k|^2<\infty.$$
 Hence, $g=\tau(h)$ with $h\in \mathcal{F}_{Slice}(\Hq)$. In particular, it shows that $g\in \mathcal{A}(\Hq)$. This completes the proof.
\end{proof}
\begin{rem}
As a direct consequence of Theorem \ref{AHcar} we have
$$\mathcal{A}(\Hq)=\lbrace \displaystyle\sum_{k=0}^\infty Q_k(q)\alpha_k; \textbf{ } (\alpha_k)_{k\geq 0}\subset \Hq \text{ and } \sum_{k=0}^{\infty}\frac{k!}{(k+1)(k+2)}|\alpha_k|^2<\infty \rbrace.$$
\end{rem}
Given $f(q)=\displaystyle \sum_{k=0}^\infty Q_k(q)\alpha_k$ and $g(q)=\displaystyle \sum_{k=0}^\infty Q_k(q)\beta_k$ in $\mathcal{A}(\Hq)$ we define their inner product by $$\scal{f,g}_{\mathcal{A}(\Hq)}:=\displaystyle \sum_{k=0}^\infty \frac{k!}{(k+1)(k+2)}\overline{\beta_k}\alpha_k,$$
so that the associated norm is $$\|f\|^2=\scal{f,f}_{\mathcal{A}(\Hq)}:=\displaystyle \sum_{k=0}^\infty \frac{k!}{(k+1)(k+2)}|\alpha_k|^2.$$
Then, one can easily check the following properties
\begin{prop}
Let $f,g,h\in\mathcal{A}(\Hq)$ and $\lambda\in\Hq$. Then, we have:
\begin{enumerate}
\item[i)] $\overline{\scal{f,g}_{\mathcal{A}(\Hq)}}=\scal{g,f}_{\mathcal{A}(\Hq)}.$
\item[ii)] $\|f\|^2=\scal{f,f}_{\mathcal{A}(\Hq)}>0$ unless $f=0$.
\item[iii)] $\scal{f,g+h}_{\mathcal{A}(\Hq)}=\scal{f,g}_{\mathcal{A}(\Hq)}+\scal{f,h}_{\mathcal{A}(\Hq)}.$
\item[iv)] $\scal{f\lambda,g}_{\mathcal{A}(\Hq)}=\scal{f,g}_{\mathcal{A}(\Hq)}\lambda$ and $\scal{f,g\lambda}_{\mathcal{A}(\Hq)}=\overline{\lambda}\scal{f,g}_{\mathcal{A}(\Hq)}.$
\end{enumerate}
\end{prop}
\begin{proof}
This statement follows using classical arguments.
\end{proof}
Now, for all $k\geq 0,$ we consider the quaternionic regular polynomials defined by \begin{equation}
T_k(q)=\displaystyle\sqrt{\frac{(k+1)(k+2)}{k!}}Q_k(q), \qquad \text{ } \forall q\in\Hq,
\end{equation}
and we introduce the following:
\begin{defn}
For all $(p,q)\in\Hq\times\Hq$, we define the function
\begin{equation}\label{circle}
G(p,q)=G_q(p):=\displaystyle \sum_{k=0}^{\infty}T_k(p)\overline{T_{k}(q)}.
\end{equation}
\end{defn}

Note that, for any $(q,p)\in\Hq\times\Hq$ we have:

\begin{enumerate}
\item[i)] $|G(p,q)|\leq \displaystyle \sum_{k=0}^\infty\frac{(k+1)(k+2)}{k!}|pq|^k<\infty.$
\item[ii)] $\overline{G(p,q)}=G(q,p).$
\item[iii)] $G(q,q)=\displaystyle \sum_{k=0}^\infty |T_k(q)|^2<\infty.$
\end{enumerate}

Let us prove that all the evaluation mappings are continuous on $\mathcal{A}(\Hq)$. Indeed, we have
\begin{prop} \label{evm}
Let $q,q'\in\Hq$, then we have:
\begin{enumerate}
\item[i)] The function $G_q:p\longmapsto G_q(p)=G(p,q)$ belongs to $\mathcal{A}(\Hq)$.
\item[ii)] The evaluation mapping $\Lambda_q: f\longmapsto \Lambda_q(f)=f(q)$ is a continuous linear functional on $\mathcal{A}(\Hq)$. Moreover, for any $f\in\mathcal{A}(\Hq)$ we have
$$ |\Lambda_q(f)|=|f(q)| \leq  \|G_q\|_{\mathcal{A}(\Hq)}\|f\|_{\mathcal{A}(\Hq)}.$$
\item[iii)] $\scal{G_{q'},G_q}_{\mathcal{A}(\Hq)}=G(q,q').$
\end{enumerate}
\end{prop}
\begin{proof}
\begin{enumerate}
\item[i)] Note that by definition of the polynomials $(T_k(q))_{k\geq 0}$, for any fixed $q\in\Hq$ we have  \begin{equation}\label{star}
    G_q(p)= \displaystyle\sum_{k=0}^\infty Q_k(p)\alpha_k(q)\ \text{ with } \ \alpha_k(q)=\frac{(k+1)(k+2)}{k!}\overline{Q_k(q)}\in\Hq, \ \ \ \forall k\geq 0.
    \end{equation}
    Moreover, observe that \begin{equation}\label{twostar}\begin{split}
 \displaystyle \|G_q\|_{\mathcal{A}(\Hq)}^2 & = \sum_{k=0}^\infty \frac{k!}{(k+1)(k+2)}|\alpha_k(q)|^2 \\
&=\sum_{k=0}^\infty \frac{(k+1)(k+2)}{k!}|Q_k(q)|^2\\
&\leq \sum_{k=0}^\infty \frac{(k+1)(k+2)}{k!}|q|^{2k}<\infty.\\
\end{split}
\end{equation}
This shows that $G_q\in\mathcal{A}(\Hq)$ for any $q\in\Hq$.
\item[ii)] If $f\in\mathcal{A}(\Hq)$, then by definition we have
 $$f(q)=\displaystyle\sum_{k=0}^\infty Q_k(q)\alpha_k \text{ and }\|f\|_{\mathcal{A}(\Hq)}^2= \sum_{k=0}^\infty \frac{k!}{(k+1)(k+2)}|\alpha_k|^2<\infty.$$
Therefore, making use of the Cauchy-Schwarz inequality we get \[ \begin{split}
 \displaystyle |\Lambda_q(f)| & = |f(q)|\\
&\leq \sum_{k=0}^\infty |Q_k(q)||\alpha_k|\\
&\leq \left(\sum_{k=0}^\infty \frac{(k+1)(k+2)}{k!}|Q_k(q)|^{2}\right)^{\frac{1}{2}} \left(\sum_{k=0}^\infty \frac{k!}{(k+1)(k+2)}|\alpha_k|^{2}\right)^{\frac{1}{2}}\\
& =\|G_q\|_{\mathcal{A}(\Hq)}\|f\|_{\mathcal{A}(\Hq)}.\\
\end{split}
\]
\item[iii)] Let $q,q'\in \Hq$ be such that
$$G_q(p)=\displaystyle \sum_{k=0}^\infty Q_k(p)\alpha_k(q) \text{ and }G_{q'}(p)=\displaystyle \sum_{k=0}^\infty Q_k(p)\alpha_k(q')$$ where we have set $\displaystyle\alpha_k(w)=\frac{(k+1)(k+2)}{k!}\overline{Q_k(w)}$ for any $w\in\Hq$ and $k\geq 0$. Therefore, we get  \[ \begin{split}
 \displaystyle \scal{G_{q'},G_q}_{\mathcal{A}(\Hq)} & = \sum_{k=0}^\infty \frac{k!}{(k+1)(k+2)}\overline{\alpha_k(q)}\alpha_k(q') \\
&=\sum_{k=0}^\infty T_k(q)\overline{T_k(q')}\\
&= G_{q'}(q)=G(q,q').\\
\end{split}
\]
\end{enumerate}
\end{proof}
As a consequence we prove the following result
\begin{thm}
The set $\mathcal{A}(\Hq)$ is a right quaternionic reproducing kernel Hilbert space whose reproducing kernel is given by the kernel function $G:\Hq\times\Hq\longrightarrow \Hq$ defined in \eqref{circle}. Moreover, for any $q\in\Hq$ and $f\in\mathcal{A}(\Hq)$ we have $$f(q)=\scal{f,G_q}_{\mathcal{A}(\Hq)}.$$
\end{thm}
\begin{proof}
According to Proposition \ref{evm} we know that all the evaluation mappings are continuous on $\mathcal{A}(\Hq)$ and $G_q\in\mathcal{A}(\Hq)$ for any $q\in\Hq$. So, we only need to prove the reproducing kernel property. Indeed, let $q\in\Hq$ and $f\in\mathcal{A}(\Hq)$ be such that $f(p)=\displaystyle\sum_{k=0}^\infty Q_k(p)\beta_k,$ for any $p\in\Hq$. Using \eqref{star} we obtain
\[ \begin{split}
 \displaystyle \scal{f,G_q}_{\mathcal{A}(\Hq)} & = \sum_{k=0}^\infty \frac{k!}{(k+1)(k+2)}\overline{\alpha_k(q)}\beta_k \\
&=\sum_{k=0}^\infty Q_k(q)\beta_k\\
&= f(q).\\
\end{split}
\]
This completes the proof.
\end{proof}
We can factorize the Bargmann-Fock-Fueter transform thanks to the following:
\begin{thm} \label{SH}
The Bargmann-Fock-Fueter transform can be realized by the commutative diagram $$\mathcal{S}_\Hq: \xymatrix{
    L^2_\Hq(\R) \ar[r] \ar[d]_{\mathcal{B}_\Hq} & \mathcal{A}(\Hq)  \\ \mathcal{F}^{2}_{Slice}(\Hq) \ar[r]_{Id} & \mathcal{SR}(\Hq) \ar[u]_{\tau}
  }$$
  so that $$\mathcal{S}_\Hq:=\tau\circ Id \circ \mathcal{B}_\Hq.$$
More precisely, for any $\varphi\in L^2_\Hq(\R)$,  and $q\in\Hq$, we have $$\mathcal{S}_\Hq\varphi(q)=\breve{f_\varphi}(q)=\displaystyle\int_{\R}\Phi(q,x)\varphi(x)dx;$$ where $$\Phi(q,x)=\displaystyle -\frac{1}{\pi^{\frac{1}{4}}}\sum_{k=0}^{\infty}\frac{Q_k(q) h_{k+2}(x)}{2^{\frac{k}{2}}k!}; \text{ }\forall(q,x)\in\Hq\times\R. $$
 \end{thm}
 \begin{proof}
 Let $\varphi\in L^2_\Hq(\R)$ and $q\in\Hq$, observe that \[ \begin{split}
 \displaystyle \mathcal{S}_\Hq [\varphi](q)& =\tau\circ Id \circ \mathcal{B}_\Hq [\varphi](q) \\
&= \int_{\R} \Delta A(q,x) \varphi(x)dx.\\
\end{split}
\]
Thus, by Proposition \ref{BFk1} we only need to prove that $$\Delta A(q,x)=\Phi(q,x) \text{ } \forall(q,x)\in\Hq\times\R,$$
where $$\Phi(q,x)=\displaystyle\int_{\C_I}K_{\mathcal{F}}(q,p)A(p,x)d\mu_I(p).$$ Indeed, note that according to Proposition 4.1 in \cite{DG1.2017} for all $(q,x)\in\Hq\times\R$ we have the following expansion of the Segal-Bargmann kernel

$$\displaystyle A(q,x)=\sum_{k=0}^{\infty}\frac{q^k }{\|q^k\|}\frac{h_k(x)}{\|h_k\|}, $$ where $\{h_k\}_{k\geq 0}$ stands for the well-known Hermite functions forming an orthogonal basis of $L^2_\Hq(\R)$.
Therefore, on the one hand we have \[ \begin{split}
 \displaystyle \Delta A(q,x) & = \frac{1}{\pi^{\frac{1}{4}}}\sum_{k=2}^\infty \frac{\Delta(q^k)h_k(x)}{k!2^{\frac{k}{2}}}\\
&=-\frac{2}{\pi^{\frac{1}{4}}}\sum_{k=2}^\infty \frac{Q_{k-2}(q)h_k(x)}{(k-2)!2^{\frac{k}{2}}}\\
&=-\frac{1}{\pi^{\frac{1}{4}}}\sum_{k=0}^\infty \frac{Q_{k}(q)h_{k+2}(x)}{k!2^{\frac{k}{2}}}.\\
\end{split}
\]
On the other hand, making use of Proposition \ref{FKex} combined with the expansion of the Segal-Bargmann kernel we get
\[ \begin{split}
 \displaystyle \Phi(q,x) & = \int_{\C_I}K_{\mathcal{F}}(q,p)A(p,x)d\mu_I(p)\\
&=-\frac{2}{\pi^{\frac{1}{4}}}\int_{\C_I}\left(\sum_{k=0}^\infty\frac{Q_k(q)}{k!}\bar{p}^{k+2}\right)\left(\sum_{j=0}^\infty\frac{p^j}{j!2^{\frac{j}{2}}}h_j(x)\right) d\mu_I(p)\\
&=-\frac{2}{\pi^{\frac{1}{4}}}\sum_{k,j=0}^\infty\frac{Q_k(q)}{k!j!}\frac{h_j(x)}{2^{\frac{j}{2}}}\scal{p^{j},p^{k+2}}_{\mathcal{F}_{Slice}(\Hq)}\\
&=-\frac{2}{\pi^{\frac{1}{4}}}\sum_{l=2,j=0}^\infty\frac{Q_{l-2}(q)}{(l-2)!j!}\frac{h_j(x)}{2^{\frac{j}{2}}}\scal{p^{j},p^{l}}_{\mathcal{F}_{Slice}(\Hq)} \\
&=-\frac{2}{\pi^{\frac{1}{4}}}\sum_{l=2,j=0}^\infty\frac{Q_{l-2}(q)}{(l-2)!}\frac{h_j(x)}{2^{\frac{j}{2}}}\delta_{l,j}\\
&=-\frac{2}{\pi^{\frac{1}{4}}}\sum_{l=2}^\infty\frac{Q_{l-2}(q)}{(l-2)!}\frac{h_l(x)}{2^{\frac{l}{2}}}\\
&=-\frac{1}{\pi^{\frac{1}{4}}}\sum_{k=0}^\infty \frac{Q_{k}(q)h_{k+2}(x)}{k!2^{\frac{k}{2}}}.\\
\end{split}
\]
This completes the proof.
\end{proof}
\begin{prop} \label{kerpp}
For all $(q,p)\in\Hq\times\Hq$ we have $$\displaystyle\int_\R\Phi(q,x)\Phi(p,x)dx=4\sum_{k=0}^{\infty}T_k(q)T_k(p).$$
\end{prop}
\begin{proof}
 Let  $(q,p)\in\Hq\times\Hq$, then
 \[ \begin{split}
 \displaystyle \int_\R\Phi(q,x)\Phi(p,x)dx & = \frac{1}{\sqrt{\pi}}\int_\R\left(\sum_{k=0}^\infty\frac{Q_k(q)h_{k+2}(x)}{2^{\frac{k}{2}}k!}\right) \left(\sum_{j=0}^\infty\frac{Q_j(p)h_{j+2}(x)}{2^{\frac{j}{2}}j!}\right) \\
&=\frac{1}{\sqrt{\pi}}\sum_{k,j=0}^\infty\frac{Q_k(q)Q_j(p)}{k!j!2^{\frac{k+j}{2}}} \int_\R h_{k+2}(x)h_{j+2}(x)dx \\
&=\frac{1}{\sqrt{\pi}}\sum_{k=0}^\infty\frac{Q_k(q)Q_k(p)}{(k!)^2 2^{k}} \Vert h_{k+2} \Vert^2.\\
\end{split}
\]
Therefore, making use of the orthogonality of Hermite functions we get
 \[ \begin{split}
 \displaystyle \int_\R\Phi(q,x)\Phi(p,x)dx & = 4\sum_{k=0}^{\infty} \frac{(k+1)(k+2)}{k!}Q_k(q)Q_k(p)\\
&=4\sum_{k=0}^{\infty}T_k(q)T_k(p).\\
\end{split}
\]
\end{proof}
\begin{rem}
Recalling that  $L^2_\Hq(\R)$ is endowed with the scalar product
$$\scal{\phi_1,\phi_2}=\displaystyle \int_\R \overline{\phi_2(x)}\phi_1(x)dx, \text{ } \forall \phi_1,\phi_2\in L^2_\Hq(\R),$$
as a consequence of Proposition \ref{kerpp} and of \eqref{circle} we get $$G(q,\overline{p})=\frac{1}{4}\scal{\Phi_p,\Phi_{\overline{q}}}, \text{ } \forall (q,p)\in\Hq\times\Hq.$$

\end{rem}
 \begin{cor}
 For all $q\in\Hq,$ the function $\Phi_q:x\longmapsto \Phi_q(x):=\Phi(q,x)$ belongs to $L^2_\Hq(\R)$ and $$\displaystyle \|\Phi_q\|_{ L^2_\Hq(\R)}=2 \left( \sum_{k=0}^\infty \frac{(k+1)(k+2)}{k!}|Q_k(q)|^2\right)^{\frac{1}{2}} < \infty. $$
 Moreover, for any $\varphi\in L^2_\Hq(\R)$ we have
 $$|\mathcal{S}_\Hq \varphi(q)| \leq\|\Phi_q\|_{ L^2_\Hq(\R)}\|\varphi\|_{ L^2_\Hq(\R)}.$$
\end{cor}
 \begin{proof}
 Let $q\in\Hq,$ then we have
   \[ \begin{split}
 \displaystyle  \|\Phi_q\|_{ L^2_\Hq(\R)}^2& = \int_\R\Phi(q,x)\overline{\Phi(q,x)}dx \\
&=\int_\R\Phi(q,x)\Phi(\bar{q},x)dx.\\
\end{split}
\] Thus, by Proposition \ref{kerpp} we get \[ \begin{split}
 \displaystyle  \|\Phi_q\|_{ L^2_\Hq(\R)}^2& = 4\sum_{k=0}^{\infty}T_k(q)T_k(\bar{q})  \\
&=4 \sum_{k=0}^\infty \frac{(k+1)(k+2)}{k!}|Q_k(q)|^2 .\\
\end{split}
\]
However, since $|Q_k(q)|^2 \leq |q|^{2k}$ for all $k\geq 0$, using \eqref{twostar} and the Cauchy-Schwarz inequality  we conclude the proof.
\end{proof}

The action of the Bargmann-Fock-Fueter transform on the normalized Hermite functions is given by
\begin{prop} \label{action}
For all $n\geq 0,$ set 
\begin{equation}\label{diesis}
\xi_n(x)=\displaystyle \frac{h_n(x)}{\|h_n\|_{ L^2_\Hq(\R)}}.
\end{equation} 
Then, we have $$\mathcal{S}_\Hq\xi_n=\breve{f}_{\xi_n}=0; \text{ for } n=0,1$$
and $$\displaystyle \mathcal{S}_\Hq\xi_n(q)=\breve{f}_{\xi_n}(q)=-2T_{n-2}(q); \text{ for all } n\geq 2.$$
 \end{prop}
  \begin{proof}
  To prove this fact we only need to use the definition of $\mathcal{S}_\Hq$ as a composition of the Fueter mapping $\tau$ and the quaternionic Segal-Bargmann transform $\mathcal{B}_\Hq$. Then, by Lemma 4.4 in \cite{DG1.2017} we know that $$\mathcal{B}_{\Hq}(\xi_n)(q)=\frac{q^n}{\sqrt{n!}}; \forall n\geq 2.$$
  Finally, we apply Remark 3.8 to conclude the proof.
\end{proof}
Then, we have
 \begin{prop} The Bargmann-Fock-Fueter transform $$\mathcal{S}_\Hq:L^2_\Hq(\R) \longrightarrow \mathcal{A}(\Hq)$$ is a quaternionic right linear bounded surjective operator such that for any $\varphi\in L^2_\Hq(\R)$, we have $$\|\mathcal{S}_\Hq \varphi\|_{\mathcal{A}(\Hq)} \leq 2 \|\varphi\|_{ L^2_\Hq(\R)}.$$
\end{prop}
 \begin{proof}
 Let $\varphi\in L^2(\R)$. Since $(\xi_k)_k$ as in \eqref{diesis} form an orthonormal basis of $L^2(\R)$ we then have $$\varphi=\displaystyle \sum_{k=0}^\infty \xi_k\alpha_k \text{ with } (\alpha_k)_k\subset\Hq \text{ and such that } \Vert \varphi \Vert^2=  \sum_{k=0}^\infty |\alpha_k|^2<\infty.$$
 Hence, since $\mathcal{B}_\Hq$ is an isometric isomorphism we use Proposition \ref{action} to get  $$\mathcal{S}_\Hq \varphi(q)=\displaystyle \sum_{k=0}^\infty  Q_k(q)\beta_k \text{ where } \beta_k=-2\sqrt{\frac{(k+1)(k+2)}{k!}}\alpha_{k+2}.$$
 In particular, this implies that
 \[ \begin{split}
 \displaystyle  \|\mathcal{S}_\Hq \varphi\|_{\mathcal{A}(\Hq)} ^2& = \sum_{k=0}^\infty \frac{k!}{(k+1)(k+2)}|\beta_k|^2   \\
&= 4\sum_{k=2}^\infty |\alpha_k|^2\leq 4 \Vert \varphi \Vert^2.\\
\end{split}
\]
Finally,  $\mathcal{S}_\Hq:L^2_\Hq(\R) \longrightarrow \mathcal{A}(\Hq)$ is surjective by construction. This completes the proof.
\end{proof}
For any $k\geq 0$, we consider the subspaces of $L^2_\Hq(\R)$ defined by $$\mathcal{H}_k:=\xi_k\Hq=\{\xi_k\alpha; \alpha\in\Hq \},$$ where $\xi_k$ denote the normalized Hermite functions. It is clear that we have the orthogonal decomposition  $$\displaystyle L^2_\Hq(\R)=\oplus^{\infty}_{k=0}\mathcal{H}_k.$$ Then, we consider $\mathcal{H}=\oplus^{\infty}_{k=2}\mathcal{H}_k$ as a subspace of $L^2_\Hq(\R)$, endowed with the induced norm and prove
\begin{prop}
Let $\varphi,\psi\in \mathcal{H}$, then we have $$\scal{\mathcal{S}_{\Hq}\varphi,\mathcal{S}_{\Hq}\psi}_{\mathcal{A}(\Hq)}=4\scal{\varphi,\psi}_{\mathcal{H}}.$$ In particular, $$\|\mathcal{S}_\Hq \varphi\|_{\mathcal{A}(\Hq)}=2\|\varphi\|_{\mathcal{H}}.$$

\end{prop}
\begin{proof}
Let $\varphi=\displaystyle\sum_{k=2}^{\infty}\xi_k\alpha_k$ and $\psi=\displaystyle\sum_{k=2}^{\infty}\xi_k\beta_k$ be two functions belonging to $\mathcal{H}$. Thus, by Proposition \ref{action} we get $$\mathcal{S}_\Hq\varphi=\displaystyle\sum_{k=0}^{\infty}Q_k\alpha_k' \text{ and }\mathcal{S}_\Hq\psi=\displaystyle\sum_{k=0}^{\infty}Q_k\beta_k',$$ where we have set $$\alpha_k'=-2\sqrt{\frac{(k+1)(k+2)}{k!}}\alpha_{k+2} \text{ and }\beta_k'=-2\sqrt{\frac{(k+1)(k+2)}{k!}}\beta_{k+2}.$$ Therefore, we obtain \[ \begin{split}
 \displaystyle \scal{\mathcal{S}_\Hq \varphi,\mathcal{S}_\Hq\psi}_{\mathcal{A}(\Hq)} & = 4\sum_{k=0}^\infty \frac{k!}{(k+1)(k+2)}\overline{\beta_k'}\alpha_k' \\
&=4\sum_{k=2}^\infty \overline{\beta_k}\alpha_k\\
&=4 \scal{\varphi,\psi}_{\mathcal{H}}.\\
\end{split}
\]
Thus, in particular, for $\psi=\varphi$ we obtain $$\|\mathcal{S}_\Hq \varphi\|_{\mathcal{A}(\Hq)}=2\|\varphi\|_{\mathcal{H}}.$$
\end{proof}
Finally, we finish this section by giving some integral representations of the quaternionic regular polynomials $(Q_k)_{k\geq 0}$ in terms of the Fock-Fueter kernel $K_\mathcal{F}(q,p)$ and the Segal-Bargmann-Fueter kernel $\Phi(q,x)$, respectively. Indeed,
\begin{prop} \label{IR1}
Let $I\in \Sq$ and $q\in\Hq$. Then, we have
\begin{itemize}
\item[i)]$Q_k(q)=\displaystyle -\frac{1}{2(k+1)(k+2)}\int_{\C_I}K_\mathcal{F}(q,p)p^{k+2}d\mu_I(p)$,  $\forall k\geq 0$.
\item[ii)] $Q_k(q)=\displaystyle -\frac{1}{4\pi^{\frac{1}{4}}2^{\frac{k}{2}}(k+1)(k+2)}\int_{\R}\Phi(q,x)h_{k+2}(x)dx$, $\forall k\geq 0$.
\end{itemize}
\end{prop}
\begin{proof}
\begin{itemize}
\item[i)] When $k\geq 0$, Proposition \ref{FKex} yields
$$K_\mathcal{F}(q,p)=\displaystyle -2\sum_{k=0}^\infty\frac{Q_k(q)}{k!}\bar{p}^{k+2}, \qquad \forall (q,p)\in\Hq\times\Hq.$$

Therefore, \[ \begin{split}
 \displaystyle \int_{\C_I}K_\mathcal{F}(q,p)p^{k+2}d\mu_I(p) & = -2\sum_{j=0}^{\infty}\frac{Q_j(q)}{j!}\scal{p^{k+2},p^{j+2}}_{\mathcal{F}_{Slice}(\Hq)} \\
&=-2\frac{Q_k(q)}{k!}\|p^{k+2}\|^2_{\mathcal{F}_{Slice}(\Hq)} \\
&=-2(k+1)(k+2)Q_k(q).\\
\end{split}
\]
\item[ii)] This assertion follows reasoning in the same way we did for i) using Theorem \ref{SH} combined with the fact that Hermite functions form an orthogonal basis of $L^2_\Hq(\R).$
\end{itemize}
\end{proof}
As a consequence we have this special identity
\begin{cor}
For any $x\in\R, I\in\Sq$ and $k\geq 0$, we have $$\displaystyle \int_{\C_I}p^{k}|p|^4e^{-|p|^2+x\bar{p}}d\lambda_I(p)=\pi (k+1)(k+2)x^k,$$ where $I\in\mathbb S$ and $d\lambda_I$ is the Lebesgue measure on $\C_I$.
\end{cor}
\begin{proof}
We only need to apply Proposition \ref{IR1} combined with the expression of the Fock-Fueter kernel for $x\in\mathbb R$, which is given by
  $$K_\mathcal{F}(x,p)=-2\overline{p}^2 e^{x\overline{p}},\qquad  \forall (x,p)\in \R\times\Hq.$$
\end{proof}

\section{The slice hyperholomorphic Bergman kernel for the quaternionic unit half ball and the fractional wedge}
In this section, we compute the explicit expression of the slice hyperholomorphic Bergman kernel on the quaternionic unit half ball and the fractional wedge domain.  The case of the quarter-ball could be treated also using smilar techniques. For the study of the Bergman kernel function in the setting of monogenic or Cauchy Fueter regular functions one may consult  for example \cite{CK2005,SMVN1997}. Let $\mathbb{B}^+$ denote the quaternionic half ball defined by the conditions $q\in\mathbb{B}$ and $Re(q)>0$. For a fixed $I\in\mathbb{S}$, let $\mathbb{B}_I^+:=\mathbb{B}^+\cap\mathbb{C}_I$ be the half disk of the complex plane $\mathbb{C}_I$. Then, the classical complex Bergman space on $\mathbb{B}_I^+$ is defined by $$\mathcal{A}(\mathbb{B}^+_I):=\{f\in Hol(\mathbb{B}^+_I), \frac{1}{\pi}\int_{\mathbb{B}^+_I}|f_I(z)|^2dA(z)<\infty \}$$ where $Hol(\mathbb{B}^+_I)$ denotes the space of holomorphic functions on the half disk $\mathbb{B}^+_I$, $z=x+Iy$ and $dA(z)=dxdy$. Note that the space $\mathcal{A}(\mathbb{B}^+_I)$ is a complex reproducing kernel Hilbert space. Furthermore, its reproducing kernel $K_{\mathbb{B}^+_I}$ is obtained as the sum of the Bergman kernels of both the complex unit disk and half plane. In particular, we have
\begin{equation}
K_{\mathbb{B}^+_I}(z,w):= \frac{1}{(1-z\overline{w})^2}+\frac{1}{(z+\overline{w})^2}; \text{ }\forall (z,w)\in\mathbb{B}^+_I\times\mathbb{B}^+_I
\end{equation}
where the first term corresponds to the Bergman kernel of the unit disk $K_{\mathbb{B}_I}$ while the second one is the Bergman kernel of the complex half plane $K_{\C^+_I}$, (see, e,g., p. 812 in \cite{CK2005}). Now, let us fix an imaginary unit $I\in\mathbb{S}$ and consider on the quaternionic half ball $\mathbb{B}^+$ the set defined by \begin{equation}\label{3stars}\mathcal{A}_{Slice}(\mathbb{B}^+):=\{f\in \mathcal{SR}(\mathbb{B}^+), \frac{1}{\pi}\int_{\mathbb{B}^+_I}|f_I(p)|^2d\sigma_I(p)<\infty \}
\end{equation} where for $p=x+Iy$ we have set $d\sigma_I(p)=dxdy$.
The set $\mathcal{A}_{Slice}(\mathbb{B}^+)$ is a right quaternionic vector space and may be endowed with the inner product:
\begin{equation}
\langle f, g\rangle_{\mathcal{A}_{Slice}(\mathbb{B}^+)}:=\frac{1}{\pi}\int_{\mathbb{B}^+_I}\overline{f_I(p)}g_I(p)d\sigma_I(p).
\end{equation} Moreover, since the quaternionic half-ball is a bounded axially symmetric slice domain it turns out that $\mathcal{A}_{Slice}(\mathbb{B}^+)$ is the slice hyperholomorphic Bergman space of the second kind on $\mathbb{B}^+$. These spaces were introduced and studied in a more general setting on axially symmetric slice domains in \cite{CGS2015}. In particular we have:

\begin{prop} The set $\mathcal{A}_{Slice}(\mathbb{B}^+)$ defined in \eqref{3stars} is a right quaternionic Hilbert space which does not depend on the choice of the imaginary unit $I\in\mathbb{S}$.
\end{prop}
Note that in this framework the evaluation mapping $$\delta_q:f\longmapsto \delta_q(f)=f(q)$$ is a right quaternionic bounded linear form on $\mathcal{A}_{Slice}(\mathbb{B}^+)$ for any $q\in\mathbb{B}^+$. Moreover, the slice hyperholomorphic Bergman kernel of the second kind associated with $\mathbb{B}^+$ or slice Bergman kernel for short, is the function $$K_{\mathbb{B}^+}:\mathbb{B}^+\times\mathbb{B}^+\longrightarrow\mathbb{B}^+, \qquad (q,r)\longmapsto K_{\mathbb{B}^+}(q,r)$$ which is defined making use of the slice hyperholomorphic extension operator, i.e. \begin{align*}
K_{\mathbb{B}^+}(q,r) &:= K_{\mathbb{B}^+}^{r}(q)
\\&:= ext[K_{\mathbb{B}^+_J}^{r}(z)](q),\qquad {\rm for}\, r\in\mathbb B^+\cap\mathbb C_J, \, q\in\mathbb B^+.
\end{align*}

The next result relates the slice Bergman kernel on the quaternionic half ball to the slice Bergman kernels in the case of the quaternionic unit ball and of the half space.
\begin{thm} \label{KB+}
The slice hyperholomorphic Bergman space $\mathcal{A}_{Slice}(\mathbb{B}^+)$ is a right quaternionic reproducing kernel Hilbert space. Moreover, for all $(q,r)\in\mathbb{B}^+\times \mathbb{B}^+$ we have:
$$K_{\mathbb{B}^+}(q,r)=K_{\mathbb{B}}(q,r)+K_{\mathbb{H}^+}(q,r),$$
where $K_{\mathbb{B}}$ and $K_{\mathbb{H}^+}$ are, respectively, the slice Bergman kernels of the quaternionic unit ball and half space.
\end{thm}
\begin{proof}
The first assertion follows from the general theory.\\
 Then, let us fix $r\in\mathbb{B}^+$ such that $r$ belongs to the slice $\mathbb{C}_J$ with $J\in\mathbb{S}$. Then, we consider the function $\psi_r$ defined by $$\psi_r(q):=K_{\mathbb{B}}(q,r)+K_{\mathbb{H}^+}(q,r),\qquad \forall q\in\mathbb{B}^+.$$ Clearly $\psi_r$ belongs to $\mathcal{A}_{Slice}(\mathbb{B}^+)$ since $\mathbb{B}^+$ is contained in both $\mathbb{B}$ and $\mathbb{H}^+$ and since by definition $K_{\mathbb{B}}$ and $K_{\mathbb{H}^+}$ are the slice Bergman kernels of the quaternionic unit ball and half space. Then, we only need to prove the reproducing kernel property. Indeed, let $f\in\mathcal{A}_{Slice}(\mathbb{B}^+)$. In particular, by the Splitting Lemma we can write $f_J(z)=F(z)+G(z)J'$ for any $z\in\mathbb{B}^+_J$ with $J'\in\mathbb{S}$ is orthogonal to $J$ and $F,G:\mathbb{B}_J^+\longrightarrow\mathbb{C}_J$ belong to the complex Bergman space $\mathcal{A}(\mathbb{B}^+_J)$. Therefore, we have
\[
\begin{split}
\langle \psi_r , f\rangle_{\mathcal{A}_{Slice}(\mathbb{B}^+)}&= \int_{\mathbb{B}^+_J}\overline{\psi_r(z)}f_J(z)d\sigma_J(z)\\
&=
\left(\int_{\mathbb{B}^+_J}\overline{\psi_r(z)}F(z)d\sigma_J(z)\right)+\left(\int_{\mathbb{B}^+_J}\overline{\psi_r(z)}G(z)d\sigma_J(z) \right)  J'\\
&= \left(\int_{\mathbb{B}^+_J}\overline{K_{\mathbb{B}_J^+}(z,r)}F(z)d\sigma_J(z)\right)+\left(\int_{\mathbb{B}^+_J}\overline{K_{\mathbb{B}_J^+}(z,r)}G(z)d\sigma_J(z) \right)  J'.
\end{split}
\]

Thus, by applying the results from the classical complex setting we get
\[
\begin{split}
\langle \psi_r , f\rangle_{\mathcal{A}_{Slice}(\mathbb{B}^+)}&= F(r)+G(r)J'\\
&=
f(r).
\end{split}
\]

So, it follows that the function $\psi_r$ belongs and reproduces any element of the space $\mathcal{A}_{Slice}(\mathbb{B}^+)$ for any $r\in\mathbb{B}^+$. Hence, by the uniqueness of the reproducing kernel we get $$K_{\mathbb{B}^+}(q,r)=K_{\mathbb{B}}(q,r)+K_{\mathbb{H}^+}(q,r),\qquad \text{ } \forall (q,r)\in\mathbb{B}^+\times \mathbb{B}^+.$$ This completes the proof.
\end{proof}
The explicit expression of the slice Bergman kernel of the quaternionic half-ball is given by the following
\begin{thm} For all $(q,r)\in\mathbb{B}^+\times \mathbb{B}^+$, we have:
$$K_{\mathbb{B}^+}(q,r)=\displaystyle (1+q^2)\left[(1-q\overline{r})*(q+\overline{r})\right]^{-*2}(1+\overline{r}^2),$$

where the $*$-product is taken with respect to the variable $q$.
\end{thm}
\begin{proof}
Let $(q, r)\in \mathbb{B}^+\times\mathbb{B}^+$ and assume that $r$ belongs to a slice $\mathbb{C}_J$ . First, observe that $$K_{\mathbb{B}^+_J}(z,r)=\displaystyle \frac{(1+z^2)(1+\overline{r}^2)}{((1-z\overline{r})(z+\overline{r}))^2}\text{; } \forall z\in \mathbb{B}_J^+.$$

Let $\Phi^r:\mathbb{B}^+\longrightarrow\mathbb{H}$ be the function defined by
$$\Phi^r(q):=ext\left[\frac{1}{(1-z\overline{r})^2 (z+\overline{r})^2}\right](q)\text{; } \forall q\in \mathbb{B}^+.$$
Then, we consider the function $$\Psi^r(q)=\displaystyle (1+q^2)\Phi^r(q)(1+\overline{r}^2)\text{; } \forall q\in \mathbb{B}^+.$$

Note that,  $\Psi^r$ is slice regular on $\mathbb{B}^+$ as a multiplication of the intrinsic slice regular function $q\longmapsto 1+q^2$ with $q\longmapsto \displaystyle \Phi^r(q)(1+\overline{r}^2)$  which is also slice regular on the quaternionic half ball by construction. Moreover, for any $z\in\mathbb{B}_J^+$ we have $$\Psi^r(z)=\displaystyle \frac{(1+z^2)(1+\overline{r}^2)}{((1-z\overline{r})(z+\overline{r}))^2}=K_{\mathbb{B}^+_J}(z,r).$$
Therefore, by the Identity Principle for slice regular functions we get $$\Psi^r(q)=K_{\mathbb{B}^+}(q,r)\text{; } \forall (q,r)\in \mathbb{B}^+\times \mathbb{B}^+.$$

Finally, we use the definition of the $*$ product to see that, for all $(q, r) \in \mathbb{B}^+\times \mathbb{B}^+$ we have
$$\Phi^r(q)=(q+\overline{r})^{-*2}*(1-q\overline{r})^{-*2}=
\left[(1-q\overline{r})*(q+\overline{r})\right]^{-*2}.
$$

\end{proof}
For $I\in\Sq$, let us now consider the wedge domain defined by $$\mathcal{W}_{\C_I}^n:=\lbrace{z\in\C_I, \text{ } Re(z)>0 \text{ and } Re(\alpha^{\frac{1}{2}}z\alpha^{\frac{1}{2}})<0 \text{ with } \alpha=e^{\frac{I\pi}{n}}}\rbrace.$$
In particular, in the complex case the Bergman kernel is given in \cite{CK2005} by $$K_{\mathcal{W}_{\C_I}^n}(z,w)=(-1)^{n}n^2\frac{z^{n-1}\bar{w}^{n-1}}{(z^n-(-1)^{n}\bar{w}^n)^2}.$$
Let $\mathcal{W}_\Hq^n$ denotes the axially symmetric completion of $\mathcal{W}_{\C_I}^n$ . In the next result, we compute the quaternionic slice hyperholomorphic Bergman kernel on $\mathcal{W}_\Hq^n$:
\begin{thm} \label{Wker}
For all $(q,r)\in\mathcal{W}_\Hq^n\times\mathcal{W}_\Hq^n$, we have
$$K_{\mathcal{W}_\Hq^n}(q,r)=(-1)^{n}n^2q^{n-1}(\bar{q}^{2n}-2(-1)^{n}\bar{q}^n\bar{r}^n+\bar{r}^{2n})\bar{r}^{n-1}(\vert q\vert^2-2(-1)^{n}Re(q^n)\bar{r}^n+\bar{r}^{2n})^{-2}.$$
\end{thm}
\begin{proof}
Let $q,r\in\mathcal{W}_\Hq^n$ be such that $r$ belongs to $\C_J$ where $J\in\mathbb{S}$. Then, for $q=x+I_qy$ and $z=x+Jy$ thanks to the extension operator we have that
$$K_{\mathcal{W}_\Hq^n}(q,r)=\displaystyle \frac{1}{2}\left(K_{\mathcal{W}_{\C_J}^n}^r(z)+K_{\mathcal{W}_{\C_J}^n}^r(\bar{z})\right)+\frac{I_qJ}{2}\left(K_{\mathcal{W}_{\C_J}^n}^r(\bar{z})-K_{\mathcal{W}_{\C_J}^n}^r(z)\right).$$
Thus, using the complex case formula we get 
\[
\begin{split}
K_{\mathcal{W}_{\C_J}^n}^r(z)+K_{\mathcal{W}_{\C_J}^n}^r(\bar{z})=2(-1)^nn^2\frac{Re(z^{n-1}\bar{z}^{2n})\bar{r}^{n-1}+Re(z^{n-1})\bar{r}^{3n-1}}{\left(\vert z \vert^{2n}+\bar{r}^{2n}-2(-1)^{n}Re(z^n)\bar{r}^n\right)^2}
\\
-2(-1)^nn^2\frac{2(-1)^{n}Re(z^{n-1}\bar{z}^n)\bar{r}^{2n-1}}{\left(\vert z \vert^{2n}+\bar{r}^{2n}-2(-1)^{n}Re(z^n)\bar{r}^n\right)^2}.
\end{split}
\]
and

\[
\begin{split}
K_{\mathcal{W}_{\C_J}^n}^r(\bar{z})-K_{\mathcal{W}_{\C_J}^n}^r(z)=2(-1)^nn^2\frac{-Im(z^{n-1}\bar{z}^{2n})J\bar{r}^{n-1}-Im(z^{n-1})J\bar{r}^{3n-1}}{\left(\vert z \vert^{2n}+\bar{r}^{2n}-2(-1)^{n}Re(z^n)\bar{r}^n\right)^2}
\\
+2(-1)^nn^2\frac{2(-1)^{n}Im(z^{n-1}\bar{z}^n)J\bar{r}^{2n-1}}{\left(\vert z \vert^{2n}+\bar{r}^{2n}-2(-1)^{n}Re(z^n)\bar{r}^n\right)^2}.
\end{split}
\]
Therefore, developing the computations we obtain
\[
\begin{split}
K_{\mathcal{W}_\Hq^n}(q,r)=(-1)^{n}n^2\left(q^{n-1}\bar{q}^{2n}\bar{r}^{n-1}-2(-1)^nq^{n-1}\bar{q}^n\bar{r}^{2n-1}+q^{n-1}\bar{r}^{3n-1} \right)
\\
\times \left(\vert q \vert^{2n}+\bar{r}^{2n}-2(-1)^{n}Re(q^n)\bar{r}^n\right)^{-2}.
\end{split}
\]
Hence, we finally get $$K_{\mathcal{W}_\Hq^n}(q,r)=(-1)^{n}n^2q^{n-1}(\bar{q}^{2n}-2(-1)^{n}\bar{q}^n\bar{r}^n+\bar{r}^{2n})\bar{r}^{n-1}(\vert q\vert^2-2(-1)^{n}Re(q^n)\bar{r}^n+\bar{r}^{2n})^{-2}.$$
This completes the proof.
\end{proof}
\begin{rem}
Observe that for the case $n=1$ in Theorem \ref{Wker} the Bergman kernel function coincide with the result obtained on the quaternionic half space in \cite{CGS2015}.
\end{rem}
\section{The Bergman-Fueter transform and some of its consequences}
In this section, we study the Bergman-Fueter integral transform on different axially symmetric slice domains $U$ on the quaternions, namely we deal with the unit ball, the half space and the unit half ball. In particular, we obtain some new generating functions and integral representations of the quaternionic regular polynomials $(Q_k)_{k\geq 0}$ obtained in Section 3. We give also the sequential characterization of the range of the Fueter mapping on the slice hyperholomorphic Bergman space on the quaternionic unit ball. First, associated to $U$ we recall from \cite{CGS2011} the following
\begin{defn}[Bergman-Fueter transform associated to $U$]
  Let $f:U\longrightarrow \mathbb{H}$ be in the slice hyperholomorphic Bergman space of the second kind $\mathcal{A}_{Slice}(U)$. Then, we define the Bergman-Fueter transform of $f$ associated to $U$ to be $$\breve{f}(q):=\displaystyle\int_{U\cap\mathbb{C}_I}K_{BF}^{U}(q,r)f(r)d\sigma(r),$$
  where $K_{BF}^{U}$ is the Bergman-Fueter kernel on $U$ defined through the following formula $$K_{BF}^{U}(q,r):=\Delta K_{U}(q,r), \text{  } \forall (q,r)\in U\times U.$$
  The Laplacian $\Delta$ is taken with respect to the variable $q$ and $d\sigma(r)$ defines the restriction of the normalized Lebesgue measure on $U_I=U\cap\C_I$.
   \end{defn}

\subsection{The quaternionic unit ball case $U=\mathbb{B}$}
 In \cite{CGS2015} an explicit expression of the Fueter-Bergman kernel was obtained when $U$ is the quaternionic unit ball $\mathbb{B}$. More precisely, we have the following result originally proved in \cite{CGS2015}:
  \begin{thm} \label{FBKB}
 For all $(q,r)\in \mathbb{B}\times \mathbb{B},$ we have 
 \begin{equation}\label{a}
\begin{split}
  K_{BF}^{\mathbb{B}}(q,r)=-\displaystyle 4\left(1-2Re(q)\bar{r}+|q|^2\bar{r}^2 \right)^{-2}\bar{r}^2+2(1-2\bar{q} \bar{r}+\bar{q}\bar{r}^2) \\ \times \left(1-2Re(q)\bar{r}+|q|^2\bar{r}^2 \right)^{-3}\bar{r}^2.
\end{split}
\end{equation}
 Furthermore, if we set $$R(q,r)=\left(1-2Re(q)\overline{r}+|q|^2\overline{r}^2 \right)^{-1},$$ then  $$K_{BF}^{\mathbb{B}}(q,r)=-4\left[R(q,r)+2K_{\mathbb{B}}(q,r)\right] R(q,r)\overline{r}^2.$$
 \end{thm}
 We prove the following
 \begin{prop} \label{sekb}
 Let $(q,r)\in\mathbb{B}\times\mathbb{B},$
 we have $$K_{BF}^{\mathbb{B}}(q,r)=-2\displaystyle\sum_{k=0}^{\infty}(k+1)(k+2)(k+3)Q_k(q)\bar{r}^{k+2}.$$
\end{prop}
 \begin{proof}
 Let $(q,r)\in \mathbb{B}\times \mathbb{B},$ making use of the slice hyperholomorphic extension operator it is clear that the slice Bergman kernel on $\mathbb{B}$ is given by the series expansion
 $$\displaystyle K_{\mathbb{B}}(q,r)=\sum_{k=0}^\infty (k+1)q^k\bar{r}^k. $$ Therefore, by definition of the Bergman-Fueter kernel we obtain:
 \[
\begin{split}
K_{BF}^{\mathbb{B}}(q,r) &= \tau_q K_{\mathbb{B}}(q,r)\\
&= -2\displaystyle\sum_{k=2}^{\infty}(k-1)k(k+1)Q_{k-2}(q)\bar{r}^{k}.
\\
&=-2\displaystyle\sum_{k=0}^{\infty}(k+1)(k+2)(k+3)Q_k(q)\bar{r}^{k+2}.
\end{split}
\]
\end{proof}
As a consequence of the latter result, we obtain the following generating function associated to the quaternionic regular polynomials $(Q_k)_{k\geq 0}$:
 \begin{thm} \label{GFQ}
 For all $(q,r)\in \mathbb{B}\times\mathbb{B},$ we have $$\displaystyle\sum_{k=0}^{\infty}(k+1)(k+2)(k+3)Q_k(q)\bar{r}^{k}=2R^2(q,r)+4K_{\mathbb{B}}(q,r)R(q,r);$$
 where $$R(q,r)= \left(1-2Re(q)\overline{r}+|q|^2\overline{r}^2 \right)^{-1} \text{ and } K_{\mathbb{B}}(q,r)=(1-2\bar{q}\bar{r}+\bar{q}^2\bar{r}^2)R(q,r)^2.$$
\end{thm}
 \begin{proof}
 Note that Theorem \ref{FBKB}  gives  $$K_{BF}^{\mathbb{B}}(q,r)=-4\left[R(q,r)+2K_{\mathbb{B}}(q,r)\right] R(q,r)\overline{r}^2.$$
 This result combined with Proposition \ref{sekb} leads to $$\displaystyle\sum_{k=0}^{\infty}(k+1)(k+2)(k+3)Q_k(q)\bar{r}^{k}=2R^2(q,r)+4K_{\mathbb{B}}(q,r)R(q,r).$$
 This completes the proof.
\end{proof}
In particular, we get the following series representation
 \begin{cor}\label{qsum}
 Let $-1<q<1$ and $r\in\mathbb{B}$. Then, we have
 $$\displaystyle\sum_{k=0}^{\infty}\frac{(k+1)(k+2)(k+3)}{6}q^k\bar{r}^{k}=(1-q\bar{r})^{-4}.$$
\end{cor}
 \begin{proof}
 We only need to observe that if $q\in\R$ then for all $k\geq 0$ we have $Q_k(q)=q^k$ thanks to the identity $\sum_{j=0}^kT_j^k=1$. Moreover, since $-1<q<1$ we have $$R(q,r)=K_\mathbb{B}(q,r)=(1-q\bar{r})^{-2}.$$ Finally, the proof is concluded by making use of Theorem \ref{GFQ}.
\end{proof}
 \begin{rem}
 As a consequence of Corollary \ref{qsum}  we observe that for all $s,t>1$ we have $$\displaystyle\sum_{k=0}^{\infty}\frac{(k+1)(k+2)(k+3)}{s^kt^k}=6\left(\frac{st}{1-st}\right)^{4}.$$ Note also that using the fact $$\forall n\geq 0: \displaystyle \sum_{k=1}^n k(k+1)=\frac{n(n+1)(n+2)}{3},$$
 we have $$\displaystyle \sum_{j=1}^\infty \sum_{k=1}^jk(k+1)q^{j-1}\bar{r}^{j-1}=2(1-q\bar{r})^{-4}.$$
\end{rem}
The right Bergman-Fueter space $\mathcal{B}(\mathbb{B})$  is the range of the slice hyperholomorphic Bergman space through the Fueter mapping. Indeed, it is defined by
$$ \mathcal{B}(\mathbb{B}):=\{ \tau (f); \text{ } f\in \mathcal{A}_{Slice}(\mathbb{B})\}.$$
Then, the next result gives the sequential characterization of the Bergman-Fueter space $\mathcal{B}(\mathbb{B})$:

\begin{thm} \label{BHcar}
Let $g\in \mathcal{R}(\mathbb{B})$. Then, $g\in\mathcal{B}(\mathbb{B})$ if and only if the following conditions are satisfied:
\begin{enumerate}
\item[i)] $\forall q\in\mathbb{B}, \textbf{ }  g(q)=\displaystyle\sum_{k=0}^\infty Q_k(q)\alpha_k$ where $(\alpha_k)_{k\geq 0}\subset \Hq.$
\item[ii)]$\displaystyle \sum_{k=0}^{\infty}\frac{|\alpha_k|^2}{(k+1)^2(k+2)^2(k+3)}<\infty.$
\end{enumerate}
 \end{thm}
\begin{proof}
 First, note that by the Fueter mapping theorem we have $\mathcal{B}(\mathbb{B})\subset\mathcal{R}(\mathbb{B}).$ Let $g\in \mathcal{B}(\mathbb{B})$, thus $g=\tau(f)$ where $f\in \mathcal{A}_{Slice}(\mathbb{B})$ such that we have  $$f(q)=\displaystyle \sum_{k=0}^{\infty}q^kc_k , \text{ with } (c_k)\subset \Hq \text{ and } \Vert f \Vert^2=\sum_{k=0}^{\infty}\frac{|c_k|^2}{k+1}<\infty.$$
 However, $$\tau(1)=\tau(q)=0 \text{ and } \tau(q^k)=-2(k-1)kQ_{k-2}(q),\qquad \forall k\geq 2.$$
 Therefore, we get $$g(q)=\displaystyle\sum_{k=0}^\infty Q_k(q)\alpha_k \text{ with } \alpha_k=-2(k+1)(k+2)c_{k+2}, \qquad \forall k\geq 0.$$ Moreover, we have
 $$\displaystyle \sum_{k=0}^{\infty}\frac{|\alpha_k|^2}{(k+1)^2(k+2)^2(k+3)}=4\displaystyle \sum_{k=2}^{\infty}\frac{|c_{k}|^2}{k+1}\leq 4\Vert f \Vert^2 <\infty.$$

 Conversely, let us suppose that the conditions i) and ii) hold. Then, we consider the function $$h(q)=\displaystyle \sum_{k=2}^{\infty}q^k\beta_k , \text{ where } \beta_k=-\frac{\alpha_{k-2}}{2(k-1)k} , \qquad \forall k\geq 2.$$
 Thus, we get $g=\tau(h)$ thanks to the formula
 $$Q_{k}(q)=-\frac{\tau(q^{k+2})}{2(k+1)(k+2)},\qquad \forall k\geq 0.$$
 Moreover, note that we have
 $$\displaystyle \Vert h \Vert^2=\sum_{k=2}^{\infty}\frac{|\beta_k|^2}{k+1}=\frac{1}{4}\sum_{k=0}^{\infty}\frac{|\alpha_k|^2}{(k+1)^2(k+2)^2(k+3)}<\infty.$$
 Hence, $g=\tau(h)$ with $h\in \mathcal{A}_{Slice}(\mathbb{B})$. In particular, it shows that $g\in \mathcal{B}(\mathbb{B})$. This completes the proof.
\end{proof}
\begin{rem}
We observe that
$$\mathcal{B}(\mathbb{B}):=\lbrace f(q)=\displaystyle\sum_{k=0}^\infty Q_k(q)\alpha_k,  \forall q\in \mathbb{B},  \textbf{  } \alpha_k\in \Hq \text{;} \sum_{k=0}^{\infty}\frac{|\alpha_k|^2}{(k+1)^2(k+2)^2(k+3)}<\infty \rbrace.$$
\end{rem}
As we have seen in Section 4 for the Fock case, it is also possible to endow the Fueter-Bergman space $\mathcal{B}(\mathbb{B})$ with the inner product

$$\scal{f,g}_{\mathcal{B}(\mathbb{B})}:=\sum_{k=0}^{\infty}\frac{\overline{\alpha_k}\beta_k}{(k+1)^2(k+2)^2(k+3)},$$ for any  $f=\displaystyle\sum_{k=0}^\infty Q_k\alpha_k$ and $g=\displaystyle\sum_{k=0}^\infty Q_k\beta_k$. It is also possible to show that $\mathcal{B}(\mathbb{B})$ is a right quaternionic reproducing kernel Hilbert space whose reproducing kernel function is given by $$L(q,r):=L_r(q)=\displaystyle \sum_{k=0}^\infty (k+1)^2(k+2)^2(k+3)Q_k(q)Q_k(\bar{r}), \forall (q,r)\in \mathbb{B}\times\mathbb{B}.$$ So that, for any $f\in\mathcal{B}(\mathbb{B})$ and $p\in\mathbb{B}$ we have $$\scal{f,L_p}_{\mathcal{B}(\mathbb{B})}=f(p).$$

An integral representation of the polynomials $(Q_k)_{k\geq 0}$ on the quaternionic unit ball $\mathbb{B}$ in terms of the Bergman-Fueter kernel is given in the following:
\begin{prop}  \label{IR2}
Let $I\in \Sq$ , $q\in\mathbb{B}$ and $k\geq 0$. Then, we have
 $$Q_k(q)=\displaystyle -\frac{1}{2(k+1)(k+2)}\int_{\mathbb{B}_I}K_{BF}^{\mathbb{B}}(q,r)r^{k+2}d\sigma_I(r).$$

\end{prop}
\begin{proof}
This follows with direct computations making use of Proposition \ref{sekb}.
\end{proof}
As a result we get this special identity
\begin{cor}
For any $-1<q<1, I\in\Sq$ and $k\geq 0$, we have $$\displaystyle \int_{\mathbb{B}_I}r^{k}|r|^4 (1-q\bar{r})^{-4}d\lambda_I(r)=\frac{\pi}{6}(k+1)(k+2)q^k,$$ where $I\in\mathbb S$ and $d\lambda_I$ is the Lebesgue measure on $\mathbb{B}_I$.
\end{cor}
\begin{proof}
We only need to apply Proposition \ref{IR2} combined with the expression of the Bergman-Fueter kernel when $-1<q<1$ .

\end{proof}
\subsection{The Bergman-Fueter transform on $\mathbb{H}^+$ and $\mathbb{B}^+$ }
 The next result gives the explicit expression of the Bergman-Fueter kernel on the quaternionic half space $\mathbb{H}^+$:
 \begin{thm}
 For all $(q,r)\in \mathbb{H}^+\times \mathbb{H}^+,$ we have $$K_{BF}^{\mathbb{H}^+}(q,r)=-\displaystyle 4\left[\left(|q|^2+2Re(q)\overline{r}+\overline{r}^2 \right)^{-2}+2(\overline{q}^2+2\bar{q} \bar{r}+\overline{r}^2)\left(|q|^2+2Re(q)\bar{r}+\bar{r}^2 \right)^{-3}\right].$$ Moreover, if we set $$P(q,r):=\left(|q|^2+2Re(q)\overline{r}+\overline{r}^2 \right)^{-1},$$ then $$K_{BF}^{\mathbb{H}^+}(q,r)=-4\left[P(q,r)+2K_{\mathbb{H}^+}(q,r)\right] P(q,r).$$
 \end{thm}
\begin{proof}
First, note that by Theorem 4.4 in \cite{CGS2015} we have $$K_{\mathbb{H}^+}(q,r)=\displaystyle\frac{1}{\pi}\left(\overline{q}^2+2\bar{q} \bar{r}+\overline{r}^2\right)\left(|q|^2+2Re(q)\overline{r}+\overline{r}^2 \right)^{-2}.$$
However, the Bergman-Fueter kernel is obtained by computing the Laplacian of the slice Bergman kernel with respect to the variable $q$, so that we have $$K_{BF}^{\mathbb{H}^+}(q,r):=\Delta K_{\mathbb{H}^+}(q,r),\qquad \textbf{ }\forall (q,r)\in \mathbb{H}^+\times \mathbb{H}^+.$$ Then, direct computations using the formula of $K_{\mathbb{H}^+}(q,r)$ show that
\[
\begin{split}
\displaystyle \frac{d^2}{dx_0^2}K_{\mathbb{H}^+}(q,r)&= 2(|q|^2+2Re(q)\bar{r}+\bar{r}^2)^{-2}-16(\bar{q}+\bar{r})(x_0+\bar{r})(|q|^2+2Re(q)\bar{r}+\bar{r}^2)^{-3} \\
&
-4(\bar{q}^2+2\bar{q}\bar{r}+\bar{r}^2)(|q|^2+2Re(q)\bar{r}+\bar{r}^2)^{-3}+24(\bar{q}^2+2\bar{q}\bar{r}+\bar{r}^2)(x_0+\bar{r})^2 \\& \times (|q|^2+2Re(q)\bar{r}+\bar{r}^2)^{-4}
\end{split}
\]
and also
\[
\begin{split}
\displaystyle \frac{d^2}{dx_1^2}K_{\mathbb{H}^+}(q,r)&= -2(|q|^2+2Re(q)\bar{r}+\bar{r}^2)^{-2}+8x_1(\bar{q}i+i\bar{q}+2i\bar{r})(|q|^2+2Re(q)\bar{r}+\bar{r}^2)^{-3} \\
&
-4(\bar{q}^2+2\bar{q}\bar{r}+\bar{r}^2)(|q|^2+2Re(q)\bar{r}+\bar{r}^2)^{-3}+24x_1^2(\bar{q}^2+2\bar{q}\bar{r}+\bar{r}^2)\\& \times (|q|^2+2Re(q)\bar{r}+\bar{r}^2)^{-4}.
\end{split}
\]

Similarly we calculate $\displaystyle \frac{d^2}{dx_2^2}K_{\mathbb{H}^+}(q,r)$ and $\displaystyle \frac{d^2}{dx_3^2}K_{\mathbb{H}^+}(q,r)$. Then, with some computations, we get $$K_{BF}^{\mathbb{H}^+}(q,r)=-\displaystyle\frac{4}{\pi}\left[\left(|q|^2+2Re(q)\bar{r}+\bar{r}^2 \right)^{-2}+2(\bar{q}^2+2\bar{q} \bar{r}+\bar{r}^2)\left(|q|^2+2Re(q)\bar{r}+\bar{r}^2 \right)^{-3}\right].$$
Finally, by replacing the function $P(q,r)$ in the previous formula we obtain $$K_{BF}^{\mathbb{H}^+}(q,r)=-4\left[P(q,r)+2K_{\mathbb{H}^+}(q,r)\right] P(q,r).$$

\end{proof}

 \begin{prop}
 The Bergman-Fueter kernel $K_{BF}^{\mathbb{H}^+}(q,r)$ is Fueter regular in $q$ and slice anti-regular in $r$ on $\mathbb{H}^+$.
  \end{prop}
 \begin{proof}
 Note that on the one hand the Fueter mapping theorem implies that $K_{BF}^{\mathbb{H}^+}(q,r)$ is Fueter regular in $q$ since $K_{\mathbb{H}^+}$ is slice regular in $q$. On the other hand, the function $P^{-1}(q,r)$ is an anti-slice regular function  with real coefficients with respect to $r$ and so is the function $P(q,r)$. Finally, the result follows since $K_{\mathbb{H}^+}$ is also anti-slice regular in $r$.
  \end{proof}
Concerning the Fueter-Bergman kernel of the quaternionic half unit ball $\mathbb{B}^+$ we have the following:
\begin{thm}
 For all $(q,r)\in \mathbb{B}^+\times \mathbb{B}^+,$ the following formula holds $$K_{BF}^{\mathbb{B}^+}(q,r)=K_{BF}^{\mathbb{B}}(q,r)+K_{BF}^{\mathbb{H}^+}(q,r).$$ Furthermore, the Bergman-Fueter kernel $K_{BF}^{\mathbb{B}^+}$ is Fueter regular in $q$ and slice anti-regular in $r$ on $\mathbb{B}^+$.
 \end{thm}

\begin{proof}
For the first statement, we only need to use the result obtained in Theorem \ref{KB+} combined with the definition of the Fueter-Bergman kernel. Then, since $\mathbb{B}^+$ is contained in both of $\mathbb{B}$ and $\mathbb{H}^+$, we have that $K_{BF}^{\mathbb{B}^+}(q,r)$ is Fueter regular in $q$ and slice anti-regular in $r$ as the sum of $K_{BF}^{\mathbb{B}}(q,r)$ and $K_{BF}^{\mathbb{H}^+}(q,r)$.
\end{proof}
\noindent{\bf Acknowledgements} \\ \\
Kamal Diki acknowledges the support of the project INdAM Doctoral Programme in Mathematics and/or Applications Cofunded by Marie Sklodowska-Curie Actions, acronym: INdAM-DP-COFUND-2015, grant number: 713485.


\end{document}